 \documentclass[11pt]{article}
\usepackage{amssymb}      

\usepackage{amsmath}      

\usepackage{amsthm}      

\usepackage{amsbsy}      

\usepackage{amscd}      

\usepackage[dvips]{graphicx}

\usepackage[active]{srcltx}

\usepackage[all]{xy}

\usepackage{color}

\usepackage[
	pagebackref,  
	breaklinks=true,
	bookmarks=false,
	colorlinks,
	linkcolor=blue,
	citecolor=red,
	urlcolor=red,
]{hyperref}

\theoremstyle{plain}      

\newtheorem{theorem}{Theorem}[section]      

\newtheorem{lemma}[theorem]{Lemma}

\newtheorem{corollary}[theorem]{Corollary}      

\newtheorem{proposition}[theorem]{Proposition}

\newtheorem{definition}[theorem]{Definition}

\theoremstyle{remark}      

\newtheorem{remark}[theorem]{Remark}

\newcommand{\Q}{{\mathbb{Q}}}        

\newcommand{\Z}{{\mathbb{Z}}}

\newcommand{\C}{{\mathbb{C}}}      

\newcommand{\R}{{\mathbb{R}}}

 \newcommand{\M}{{\mathcal{M}}}

\newcommand{\q}{{\mathfrak{q}}}    
\newcommand{\A}{{\mathfrak{A}}}

\begin{document}

\date{\today}

\title{Finite quotients of symplectic groups vs mapping class groups}

\author{\begin{tabular}{cc}      Louis Funar &  Wolfgang Pitsch\footnote{Supported by the FEDER/MEC grant MTM2016-80439-P.
}\\      
\small \em Institut Fourier BP 74, UMR 5582       &\small \em Departament de Matem\`atiques \\      
\small \em University Grenoble Alpes &\small \em Universitat Aut\`onoma de Barcelona    \\      
\small \em CS 40700, 38058 Grenoble cedex 9,  France      
&\small \em 08193 Bellaterra (Cerdanyola del Vall\`es), Espana  \\      
\small \em e-mail: {\tt louis.funar@univ-grenoble-alpes.fr}      
& \small \em e-mail: {\tt pitsch@mat.uab.es} \\      
\end{tabular}      
}


\maketitle

\begin{abstract}
We give alternative computations of the Schur multiplier of 
$Sp(2g,\Z/D\Z)$, when $D$ is divisible by 4 and $g\geq 4$: a first one 
using $K$-theory arguments based on the work of Barge and Lannes and a second one based on the  Weil representations of symplectic groups arising in 
abelian Chern-Simons theory.  We can also retrieve this way  
Deligne's non-residual finiteness of the universal central extension 
$\widetilde{Sp(2g,\Z)}$.
We prove then that the  image of the second homology into 
finite quotients of symplectic groups over a  
Dedekind domain of arithmetic type are torsion groups 
of uniformly bounded size. In contrast, quantum representations 
produce for every prime $p$, finite quotients of 
the mapping class group of genus $g\geq 3$ 
whose second homology image has $p$-torsion. 
We further derive that all central extensions of 
the mapping class group are residually finite and deduce that 
mapping class groups have Serre's property $A_2$ for 
trivial modules, contrary to symplectic groups. 
Eventually we compute the module of coinvariants  
$H_2(\mathfrak{sp}_{2g}(2))_{Sp(2g,\Z/2^k\Z)}=\Z/2\Z$.

\vspace{0.1cm}



\end{abstract}

\section{Introduction and statements}
Let $\Sigma_{g,k}$ denote a connected oriented surface of genus $g$ 
with $k$ boundary components and  $M_{g,k}$ be its mapping class group, 
namely the group of isotopy classes of 
orientation preserving homeomorphisms that fix point-wise 
the boundary components. If $k=0$, we simply write $M_g$ for $M_{g,0}$. 
The action of $M_{g}$ on the integral homology of $\Sigma_{g}$ 
equipped with some symplectic basis gives a surjective 
homomorphism $M_{g} \rightarrow Sp(2g,\Z)$, and it is a natural 
and classical problem to compare the properties of these two groups.  
The present paper is concerned with the central extensions and 
2-homology groups of these two groups and their finite quotients and is a sequel 
to \cite{Fu3} and \cite{FP}.

\vspace{0.2cm}
\noindent 
Our first result is:  
\begin{theorem}\label{tors-sympl0}
The second homology group of finite principal congruence quotients 
of $Sp(2g,\Z)$, $g\geq 4$ is 
\[H_2(Sp(2g,\Z/D\Z))=\left\{\begin{array}{ll} 
\Z/2\Z,  & {\rm if }\; D\equiv 0 \; ({\rm mod} \; 4), \\
0, & {\rm otherwise}.
\end{array}\right.
\]  
\end{theorem}

\vspace{0.2cm}
\noindent 
This result when $D$ is not divisible by 4 is an old theorem 
of Stein (see \cite[Thm. 2.13 and Prop. 3.3.a]{Stein}) while  
the case $D\equiv  0 \; ({\rm mod} \; 4)$ remained open for a while,  
as mentioned in \cite[remarks after Thm. 3.8]{Pu1},  
because the condition  $D\not\equiv 0 \; ({\rm mod} \; 4)$ 
seemed essential for all results in there.  
A short proof using geometric group theory was obtained by the authors in \cite{FP} and 
another proof was independently obtained in \cite{BCRR} (see also \cite{Be}). 
One of our aims here is to present alternative proofs based on mapping class group representations arising in the $U(1)$ Chern-Simons theory and $K$-theory, respectively.

\vspace{0.2cm}
\noindent 
The analogous result for special linear groups has long been known. The equality   $H_2(SL(2,\Z/D\Z))=\Z/2\Z$, for 
$D\equiv 0({\rm mod }\; 4)$ was proved by 
Beyl (see \cite{Beyl}) and for large $n$  
Dennis and Stein proved using $K$-theoretic methods that 
$H_2(SL(n,\Z/D\Z))=\Z/2\Z$, for 
$D\equiv 0({\rm mod }\; 4)$, while $H_2(SL(n,\Z/D\Z))=0$, otherwise, 
see \cite[Cor. 10.2]{DS1}. 

\vspace{0.2cm}
\noindent 
Our main motivation  for carrying the computation of Theorem \ref{tors-sympl0} 
was to better understand the (non-)residual finiteness of central extensions. 
The second result of this paper is the following:
\begin{theorem}\label{nonresid0} 
The universal central extension $\widetilde{Sp(2g,\Z)}$ is 
not residually finite when $g\geq 4$ since 
the image of the center under any homomorphism into 
a finite group has order at most two.
Moreover, the image of 
the center has order two under the 
natural homomorphism of $\widetilde{Sp(2g,\Z)}$  
into the universal central extension of $Sp(2g,\Z/D\Z)$, 
where $D$ is a multiple of $4$ and $g\geq 4$.  
\end{theorem}
\noindent 
The first part of this result is the statement of 
{\em Deligne's non-residual finiteness 
theorem} from \cite{De}, which was stated for $g\geq 2$. 
In what concerns the sharpness statement, 
Putman in \cite[Thm. F]{Pu1} has previously obtained the existence 
of finite index subgroups of $\widetilde{Sp(2g,\Z)}$ which contain 
$2\Z$ but not $\Z$. We provide some explicit constructions of such finite index 
normal subgroups. 
The relation between Theorems \ref{tors-sympl0} and \ref{nonresid0}
is somewhat intricate. For instance, the statement 
$H_2(Sp(2g,\Z/D\Z))\in\{0,\Z/2\Z\}$, for $g\geq 4$ is a consequence 
of Deligne's theorem. This statement and the second part 
of Theorem \ref{nonresid0} actually imply Theorem \ref{tors-sympl0} and 
this is our first proof of the latter. 
However, we can reverse all implications and using now a different proof 
of Theorem \ref{tors-sympl0}, based on $K$-theory arguments,  
we derive from it another proof of Theorem 
\ref{nonresid0}. In particular, this provides a new proof of Deligne's 
theorem, independent of Moore's theory of topological central extensions 
from  \cite{Moore}. 
 
\begin{remark}
Deligne proved that the image of twice the generator of the center of $\widetilde{Sp(2g,\Z)}$ 
under any homomorphism into a finite group is trivial, for any $g\geq 2$. 
Moreover, the intersection of the finite index subgroups of $\widetilde{Sp(2g,\Z)}$ 
is precisely the subgroup generated by twice the center generator, when $g\geq 4$. 
Theorem \ref{nonresid0} provides explicit morphisms into finite groups for which 
the generator of the center maps into a nontrivial element.   
\end{remark}

\begin{remark}
Note that $H_2(Sp(6,\Z))=\Z\oplus \Z/2\Z$, according to \cite{Stein3}, while $Sp(4,\Z/4\Z)$ is not perfect. Thus the central extension by $\Z$ considered in \cite{De} is not the 
universal central extension of $Sp(2g,\Z)$, when $g\in\{2,3\}$.   
The computation of the Schur multiplier for small $g$ was completed in \cite{BCRR}:  
$H_2(Sp(2g,\Z/D\Z))= \Z/2\Z\oplus \Z/2\Z$, for $g\in\{2,3\}$ and 
$D\equiv 0 \; ({\rm mod} \; 4)$. This corrects a misprint in \cite{FP}, where 
for $g=3$ we only proved that the Schur multiplier is nontrivial. 
\end{remark}

\vspace{0.2cm}
\noindent 
The theorem stated in \cite{De} is much more general 
and covers higher rank Chevalley groups over number fields. 
The proof of Theorem \ref{nonresid0} also shows that 
the residual finiteness of central extensions is directly related 
to the torsion of the second homology of finite quotients 
of the group. Then the general form of Deligne's theorem can be used 
to obtain bounds for the torsion arising in the second homology of  finite quotients 
of symplectic groups over more general rings. 
First, we have: 

\begin{theorem}\label{boundtorsion0}
Let $\A$ be the ring of $S$-integers of a number field which is 
not totally imaginary and $g \geq 3$ an integer. 
Then there is a uniform bound (independent of $g$ and $F$) 
for the order of the torsion group
$p_*(H_2(Sp(2g,\A))) \subseteq H_2(F)$
for any surjective homomorphism $p:Sp(2g,\A)\to F$ onto a finite group. Here 
$p_*:H_2(Sp(2g,\A))\to H_2(F)$ denotes the map induced in homology and all homology groups are considered with (trivial) integral coefficients. 
\end{theorem}
\noindent 
The result is a rather immediate consequence of the general 
Deligne theorem from \cite{De}, 
along with classical results of Borel and Serre \cite{BoSe} 
and Bass, Milnor and Serre \cite{BMS} on the congruence subgroup problem.

\begin{remark}
With slightly more effort we can show that this holds also when 
the number field is totally imaginary due to the finiteness 
of the congruence kernel.  Furthermore, the result holds for any Chevalley group 
instead of the symplectic group, with a similar proof. 
More generally it holds under the conditions of \cite{De}, namely 
for every absolutely simple  simply connected algebraic group  $\mathbb G$ 
over a number field $\mathbb K$, $S$ a finite set of places of $\mathbb K$ 
containing all archimedean ones and such that 
$\sum_{v\in S}{\rm rank}\: \mathbb G(\mathbb K_v) \geq 2$ and 
$\A$ the associated ring of $S$-integers. The quasi-split assumption 
in \cite{De} was removed in \cite{Rag}.  
However, for the sake of simplicity we will only consider symplectic groups 
in the sequel. 
\end{remark}

\vspace{0.2cm}\noindent 
Theorem \ref{boundtorsion0} contrasts with the abundance of  
finite quotients of mapping class groups:
 
\begin{theorem}\label{res0}
For any prime $p$ and $g\geq 3$ there exist surjective homomorphisms $\pi:M_g\to F$ 
onto finite groups $F$  
such that  $\pi_*(H_2(M_g)) \subseteq H_2(F)$
has $p$-torsion elements, and in particular is not trivial.
\end{theorem}
\noindent 
We prove this result  by exhibiting explicit finite quotients 
of the universal central extension of a mapping class group that arise 
from the so-called quantum representations. We refine here the 
approach in \cite{Fu3} where the first author proved that central 
extensions of $M_g$ by $\Z$ are residually finite. In the meantime, it was  proved in \cite{F4,MR} by more sophisticated tools that the set of  
quotients of mapping class groups contains arbitrarily large rank 
finite groups of Lie type. Notice however that the  family of 
quotients obtained in Theorem \ref{res0} are  different in nature than those 
obtained in \cite{F4,MaR}, although their source is the same (see Proposition 
\ref{unequal} for details).  

\vspace{0.2cm}\noindent 
Theorem \ref{res0} shows that in the case of 
non-abelian quantum representations of mapping class groups there is 
no finite central extension for which all projective representations 
could be lifted to linear representations, when the genus is $g\geq 2$ 
(see Corollary \ref{nontrivial} for the precise statement).

\vspace{0.2cm}\noindent 
When $G$ is a discrete group we denote by $\widehat{G}$ its {\em profinite 
completion}, i.e. the projective limit of the directed system of all 
its finite quotients. There is a natural homomorphism $i:G\to \widehat{G}$ 
which is injective if and only if $G$ is residually finite. 
A {\em discrete $\widehat{G}$-module} is an abelian group endowed with a 
{\em continuous} action of $\widehat{G}$. We will simply call them 
$\widehat{G}$-modules in the sequel. We say that a $\widehat{G}$-module is {\em trivial} if  
the $\widehat{G}$-action is trivial.  Recall 
following \cite[I.2.6]{Serre} that:
\begin{definition}
A  discrete group $G$ has 
property $A_n$ for the finite $\widehat{G}$-module $M$ if the homomorphism
$H^k(\widehat{G}, M)\to H^k(G,M)$ is an isomorphism for
$k\leq n$ and injective for $k=n+1$. 
Furthermore $G$ is called {\em good} if it has property $A_n$ for all $n$ 
and for all finite $\widehat{G}$-modules.
\end{definition}
\noindent 
It is known, for instance, that all groups have property $A_1$.

\vspace{0.2cm}\noindent
Now, Deligne's theorem  
on the non-residual finiteness
of the universal central extension of $Sp(2g,\Z)$ 
actually is equivalent to the fact that 
$Sp(2g,\Z)$ has not property $A_2$ for the trivial 
$Sp(2g,\Z)$-modules (see also \cite{GJZ}).

\vspace{0.2cm}\noindent
Our next result is: 
\begin{theorem}\label{a2}
For $g \geq 4$ the mapping class group $M_g$ has property $A_2$ 
for the trivial $\widehat{M_g}$-modules.
\end{theorem}
\vspace{0.2cm}\noindent
Our proof also yields the following: 
\begin{corollary}\label{vtrivialmcg}
Central extensions of $M_g$, $g\geq 4$, by finite abelian groups are virtually trivial and 
can be obtained as pull-backs from central extensions of finite quotients of $M_g$. 
\end{corollary}

\vspace{0.2cm}\noindent
The last part of this article is devoted to a partial extension of the method 
used by Putman in \cite{Pu1} to compute $H_2(Sp(2g,\Z/D\Z))$, when $D\equiv 0\; ({\rm mod}\; 4)$, 
using induction. One key point is to show that there is a 
potential $\Z/2\Z$ factor that appears for 
$H_2(Sp(2g,\Z/4\Z))$. Although we couldn't complete the 
proof of Theorem \ref{tors-sympl0} this way, 
our main result in this direction may be of independent interest. 
Set $\mathfrak{sp}_{2g}(p)$ for  the additive group of those  
$2g$-by-$2g$ matrices $M$ with entries in $\Z/p\Z$ that satisfy the 
equation $M^{\top} J_g + J_g M \equiv 0 ({\rm mod } \; p)$, where 
$J_g  = \left( \begin{matrix} 
0 & \mathbf 1_g \\ 
-\mathbf 1_g & 0      
\end{matrix}
\right)$ is the symplectic form. Then $\mathfrak{sp}_{2g}(p)$ is endowed with 
a natural  $Sp(2g,\Z/p\Z)$-action.

\begin{theorem}\label{lem H2sp(p)}
For any integers $g \geq 4$, $k\geq 1$ and prime $p$,  the space of 
co-invariants in homology is:
 \[
H_2(\mathfrak{sp}_{2g}(p))_{Sp(2g,\Z/p^k\Z)} = \left\{  \begin{array}{ll} 0, & \text{if } p \text{ is odd,} \\ \Z/2\Z, & \text{if } p= 2. \end{array}\right.
 \]
\end{theorem}
\noindent 
An immediate consequence of this result is the alternative 
$H_2(Sp(2g,\Z/4\Z))\in\{0,\Z/2\Z\}$, without use of Deligne's theorem.

\vspace{0.2cm}\noindent 
The plan of this article is the following. 

\vspace{0.2cm}\noindent 
In Section \ref{compute} we prove Theorem \ref{tors-sympl0}.  
Although it is easy to show that the groups $H_2(Sp(2g,\Z/2^k\Z))$ 
are cyclic for $g\geq 4$, their non-triviality is much more involved. That this group 
is trivial for $k=1$ is a known fact, for instance by Stein's  
results \cite{Stein2}. 
We give two  different proofs of the non-triviality, each one of them having 
its advantages and disadvantages in terms of bounds for detections 
or sophistication. 
The first proof is $K$-theoretical in nature and uses a generalization 
of Sharpe's exact sequence relating $K$-theory to symplectic $K$-theory due 
to Barge and Lannes \cite{BaLa}. Indeed, by the stability results, this 
$\Z/2\Z$ should correspond to a class in $KSp_2(\Z/4\Z)$. There is a 
natural map from this group to a Witt group of symmetric non-degenerate 
bilinear forms on free $\Z/4\Z$-modules, and it turns out that the class 
is detected by the class of the bilinear map of 
matrix $\left( \begin{matrix} 2 & 1 \\ 1 & 2  \end{matrix} \right)$. 
The second proof uses mapping class groups. We show that there 
is a perfect candidate to detect this $\Z/2\Z$ that comes from a Weil 
representation of the symplectic group. This is, by construction,  a 
representation of $Sp(2g,\Z)$ into a projective unitary group that factors 
through $Sp(2g,\Z/4n\Z)$. To show that it detects the factor $\Z/2\Z$ it is 
enough to show that this representation does not lift to a linear 
representation.  We will show that the pull-back of the representation 
on the mapping class group $M_g$ does not linearize. This proof 
relies on deep results of Gervais \cite{Ge}.
The projective representation that we use is related to the theory of 
theta functions on symplectic groups, this relation is explained in 
an appendix to this article.

\vspace{0.2cm}\noindent 
In Section~\ref{residual} we first state the relation 
between the torsion in the second homology of a perfect group and 
residual finiteness of its universal central extension. 
We then prove Theorem \ref{boundtorsion0} for Dedekind 
domains by analyzing Deligne's central extension. We further  
specify our discussion to the group $Sp(2g,\Z)$, and show 
how the result stated in Theorem \ref{tors-sympl0} allows
 to show that Deligne's result is sharp.

\vspace{0.2cm}\noindent 
Finally in Section \ref{mcg} we discuss the  
case of the mapping class groups and prove Theorem \ref{res0} 
and Theorem \ref{a2} using the quantum representations that arise from 
the $SU(2)$-TQFT's. These representations are the non-abelian counterpart 
of the Weil representations of symplectic groups, which might be described 
as the quantum representations that arise from the $U(1)$-TQFT. 

\vspace{0.2cm}\noindent 
Finally, in appendix A we give a small overview of the 
relation between Weil representations and extensions of the symplectic group.

\vspace{0.2cm}\noindent 
In all this work, unless otherwise specified, all (co)homology groups 
are with coefficients in $\Z$, and we drop it from the notation so 
that for a group $G$, $H_\ast(G) = H_\ast(G;\Z)$ and $H^\ast(G) = H^\ast(G;\Z)$.

\vspace{0.2cm}\noindent 
{\bf Acknowledgements.} We are thankful to  Jean Barge, Nicolas Bergeron, 
Will Cavendish, Florian Deloup, Philippe Elbaz-Vincent, Richard Hain, Pierre Lochak, Greg McShane, Ivan Marin, Gregor Masbaum, Andy Putman,  Alexander Rahm and Alan Reid for helpful 
discussions and suggestions. We are grateful to the referee for his/her thorough
screening of the manuscript and  for the comments which considerably improved the presentation. 

\section{Proof of Theorem \ref{tors-sympl0}}\label{compute}
\subsection{Preliminaries}
Let $D=p_1^{n_1}p_2^{n_2}\cdots p_s^{n_s}$ be the prime decomposition of 
an integer $D$.  Then, according to  \cite[Thm. 5]{NS}  we have 
$Sp(2g,\Z/D\Z)=\oplus_{i=1}^s Sp(2g,\Z/p_i^{n_i}\Z)$.
Since symplectic groups are perfect for $g\geq3$, see e.g. 
\cite[Thm. 5.1]{Pu1},  from the K\"unneth formula, we  derive: 
\[
H_2(Sp(2g,\Z/D\Z))=\oplus_{i=1}^s H_2(Sp(2g,\Z/p_i^{n_i}\Z)).
\] 
Then, from Stein's computations for $D\not\equiv 0\; ({\rm mod}\; 4)$,  
see \cite{Stein,Stein2}, Theorem \ref{tors-sympl0} is 
equivalent to the  statement: 
\[ H_2(Sp(2g,\Z/2^k\Z)) = \Z/2\Z, \; {\rm for \; all \; } g \geq 4, k \geq 2.  
\] 
We will freely use in the sequel two classical results due to Stein.   
{\em Stein's isomorphism theorem}, see \cite[Prop. 3.3.(a)]{Stein},  
states that there is an isomorphism:  
\[
H_2(Sp(2g,\Z/2^k\Z))\simeq H_2(Sp(2g,\Z/2^{k+1}\Z)), \; 
{\rm for \; all \; } g \geq 3, k \geq 2. 
\]
Further, {\em Stein's stability theorem} from \cite[Thm. 2.13]{Stein} states that 
the stabilization homomorphism $Sp(2g,\Z/2^k\Z)\hookrightarrow 
Sp(2g+2,\Z/2^k\Z)$ induces an isomorphism: 
\[
H_2(Sp(2g,\Z/2^k\Z))\simeq H_2(Sp(2g+2,\Z/2^{k}\Z)), \; 
{\rm for \; all \; } g \geq 4, k \geq 1. 
\]
Therefore, to prove Theorem \ref{tors-sympl0}  it suffices to show that: 
\[ H_2(Sp(2g,\Z/2^k\Z))=\Z/2\Z, \; {\rm for \; some \; } g \geq 4, k \geq 2. 
\]
We provide hereafter two different proofs of this statement, each having its 
own advantage.

\vspace{0.2cm}\noindent   
The first proof, based on an extension of 
Sharpe's sequence in symplectic $K$-theory due to Barge and Lannes,   
see \cite{BaLa}, gives the result already for $Sp(2g,\Z/4\Z)$. 
Moreover, this proof does not rely on  Deligne's theorem.

\vspace{0.2cm}\noindent   
For the second proof, 
the starting point is the following intermediary result: 
\begin{proposition}\label{alternative}
We have $H_2(Sp(2g,\Z/2^k\Z))\in \{0, \Z/2\Z\}$, when $g\geq 4$ and $k \geq 2$.
\end{proposition}

\vspace{0.2cm}\noindent 
This was obtained in 
\cite[Prop. 3.1]{FP} as an immediate consequence of Deligne's theorem. 
A direct proof of the alternative  
$H_2(Sp(2g,\Z/4\Z))\in \{0, \Z/2\Z\}$ when $g\geq 4$ will be given 
in Section \ref{motiv}, under the form of 
Corollary \ref{inductive} of Theorem \ref{lem H2sp(p)}. 
When $g=3$ our arguments only provide a weaker statement, namely that  
 $H_2(Sp(6,\Z/2^k\Z))\in \{0, \Z/2\Z,\Z/2\Z\oplus\Z/2\Z\}$.

\vspace{0.2cm}\noindent 
Then it will be enough to find  
a non-trivial extension of $Sp(2g,\Z/2^k\Z)$ by $\Z/2\Z$ for some 
$g\geq 4$, $k\geq 2$. 

\vspace{0.2cm}\noindent   
Our second  proof seems more elementary and provides  
an explicit non-trivial central extension of $Sp(2g,\Z/4n\Z)$ by $\Z/2\Z$, for all 
integers $n\geq 1$. Moreover, it does not use 
Stein's isomorphism theorem and relies instead on the study of 
the Weil representations of symplectic groups, or equivalently 
abelian quantum representations of mapping class groups.
Since these representations come from theta functions  
this approach is deeply connected to Putman's  
approach. In fact the proof of the Theorem F in \cite{Pu1} is based   
on his Lemma 5.5 whose proof 
needed the transformation formulas for the classical theta nulls.


\subsection{A $K$-theory computation of  $H_2(Sp(2g, \Z/4\Z) )$}
The proof below uses slightly 
more sophisticated tools which were developed by Barge and 
Lannes in \cite{BaLa} and allow us to bypass 
Deligne's theorem. According to Stein's stability theorem 
\cite{Stein} it is enough to prove that 
$H_2(Sp(2g,\Z/4\Z))=\Z/2\Z$, for $g$ large.
It is well-known that the second homology of the 
linear and symplectic groups can be interpreted in terms 
of the $K$-theory group $K_2$. Denote by $K_1(\A),K_2(\A)$ 
and $KSp_1(\A)$, $KSp_2(\A)$  the groups of algebraic $K$-theory of 
the stable linear groups and symplectic groups over the 
commutative ring $\A$, respectively. See \cite{HO} for definitions. 
Our claim is equivalent to the fact that 
$ KSp_2(\Z/4\Z)=\Z/2\Z$. 

\vspace{0.2cm}\noindent  For an arbitrary ring $R$, the group $
V(R)$  is defined as follows (see \cite[Section 4.5.1]{BaLa}). 
Consider the set of triples $(L;q_0,q_1)$, where $L$ is a free $R$-module of 
finite rank and $q_0$ and $q_1$ are non-degenerate symmetric bilinear forms. 
Two such triples $(L;q_0,q_1)$ and $(L';q'_0,q'_1)$ are equivalent, if there 
exists an $R$-linear isomorphism $a:L \rightarrow L'$ such that 
$ a^\ast \circ q'_0 \circ a = q_0$ and $a^\ast \circ q'_1 \circ a = q_1$. 
Under orthogonal sum these triples form a monoid. The group  $V(R)$ is by definition the 
quotient of the Grothendieck group associated to the monoid of such triples  by the subgroup generated 
by Chasles' relations, that is the subgroup generated by the elements of 
the form:
\[
[L;q_0,q_1] + [L;q_1,q_2] - [L;q_0,q_2]. 
\]
  
\vspace{0.2cm}\noindent  
Our key ingredient is the exact sequence from \cite[Thm. 5.4.1]{BaLa}, 
which is a generalization of Sharpe's exact sequence (see \cite[Thm. 5.6.7]{HO}) 
in $K$-theory: 

\begin{equation}
K_2(\Z/4\Z)\to KSp_2(\Z/4\Z) \to V(\Z/4\Z)\to K_1(\Z/4\Z)\to 1.
\end{equation}
\vspace{0.2cm}\noindent  
We first show:

\begin{lemma}\label{rezs} 
The homomorphism 
$K_2(\Z/4\Z)\to KSp_2(\Z/4\Z)$ is trivial. 
\end{lemma}
\begin{proof}[Proof of Lemma \ref{rezs}]
Recall from \cite{BaLa} that this homomorphism 
is induced by the hyperbolization inclusion 
$SL(g,\Z/4\Z)\to Sp(2g,\Z/4\Z)$, 
which sends the matrix $A$ to $
A\oplus (A^{-1})^{\top}$.  By stability of the homology groups of the special linear and symplectic  groups it is enough to show that the induced map
\[
H: H_2(SL(g,\Z/4\Z)) \rightarrow H_2(Sp(2g,\Z/4\Z))
\]
is trivial for $g \geq 5$.

\vspace{0.2cm}\noindent 
Together with the hyperbolization, mod $4$ reduction of the coefficients gives us a commutative diagram:
\[
\xymatrix{
H_2(SL(g,\Z)) \ar[d] \ar[r]^- h &  H_2(Sp(2g,\Z)) \ar[d] \\
H_2(SL(g,\Z/4\Z)) \ar[r]^- H &  H_2(Sp(2g,\Z/4\Z)) 
}
\]

\vspace{0.2cm}\noindent 
It is known that $H_2(SL(g,\Z))=\Z/2\Z$, when $g\geq 5$ and $H_2(Sp(2g,\Z))=\Z$, when $g\geq 4$, 
see e.g. \cite{Mat} and \cite[Thm. 10.1, Thm. 5.10, Remark]{Milnor}. 
Moreover, the homomorphism induced in homology $H_2(SL(g,\Z)) \rightarrow H_2(SL(g,\Z/4\Z))$ by the reduction mod $4$ of coefficients is surjective, see \cite[section 10 pp. 92]{Milnor}. 
Alternatively, we can infer it from the proof \cite[Prop. 3.2]{FP}. 
Actually, Dennis proved that this map is an isomorphism and hence both groups are isomorphic to $\Z/2\Z$, see \cite{DS2}. It follows that the hiperbolization map $h$ is trivial, 
and so is $H$.

\end{proof}

\begin{remark}
An alternative argument is as follows.   
The hyperbolization homomorphism $H: K_2(\Z/4\Z)\to KSp_2(\Z/4\Z)$  
sends the Dennis-Stein symbol $\{r,s\}$ to the Dennis-Stein 
symplectic symbol $[r^2,s]$, see e.g. \cite[Paragraph 5.6.2]{HO}.  
According to \cite[Prop. 3.3 (b)]{Stein} the group $K_2(\Z/4\Z)$ is generated 
by $\{-1,-1\}$ and thus its image by $H$ is generated by 
$[1,-1]=0$.
\end{remark}

\vspace{0.2cm}
\noindent
Going back to Sharpe's exact sequence, 
it is known that: 
\begin{equation} 
K_1(\Z/4\Z)\cong (\Z/4\Z)^*\cong \Z/2\Z,
\end{equation}
and the problem is to compute the discriminant map $V(\Z/4\Z) \rightarrow K_1(\Z/4\Z)$.

\vspace{0.2cm}
\noindent
Recall from \cite[Section 4.5.1]{BaLa} that there is a canonical map from $V(R)$ to  the Grothendieck-Witt group of 
symmetric non-degenerate bilinear forms   over free modules that 
sends the class $[L;q_0,q_1]$ to $q_1 - q_0$. Since 
$\Z/4\Z$ is a local ring, we know  that $SK_1(\Z/4\Z)=1$ and hence by
 \cite[Corollary 4.5.1.5]{BaLa} we have a pull-back square of abelian groups:
\[
 \begin{array}{ccc}
  V(\Z/4\Z) & \longrightarrow & I(\Z/4\Z) \\
  \downarrow & & \downarrow \\
  (\Z/4\Z)^\ast & \longrightarrow & (\Z/4\Z)^\ast/((\Z/4\Z)^\ast)^2,
 \end{array}
\]
where $I(\Z/4\Z)$ is a the augmentation ideal of the Grothendieck-Witt ring 
of $\Z/4\Z$. But $(\Z/4\Z)^\ast =\{1,3\}$, and only  $1$ is a square, hence the bottom  arrow in the  square is an isomorphism 
$\Z/2\Z \cong \Z/2\Z$. Thus $V(\Z/4\Z) \cong I(\Z/4\Z)$ and the kernel of 
$V(\Z/4\Z) \rightarrow (\Z/4\Z)^\ast \cong K_1(\Z/4\Z)$ is the kernel of 
the discriminant homomorphism 
$I(\Z/4\Z) \rightarrow (\Z/4\Z)^\ast/((\Z/4\Z)^\ast)^2$. 
To compute  $V(\Z/4\Z)$ it is therefore enough to compute  the 
 Witt ring $W(\Z/4\Z)$. Recall that this is the 
quotient of the monoid of symmetric non-degenerate bilinear forms 
over finitely  generated projective modules 
modulo the sub-monoid of \emph{split} 
forms. A bilinear form is split if the underlying free module  contains a 
direct summand $N$ such that $N = N^\bot$. By a classical result of 
Kaplansky,  finitely generated projective modules over $\Z/4\Z$ are free.  
Then, from  \cite[Lemma 6.3]{MH} any split form can be written in matrix form as:
\[ 
\left(
\begin{matrix}
0 & \mathbf 1 \\
\mathbf 1 & A
\end{matrix}
\right),
\]
for some symmetric matrix $A$, where $\mathbf 1$ denotes the identity matrix. Isotropic 
submodules form an inductive system, 
and therefore any isotropic submodule is contained in a maximal one.  
These have all the same rank. In the case of a split form this rank is 
necessarily half of the rank of the underlying free module, which is 
therefore even. The main difficulty in the following computation is 
due to the fact that  $2$ is not a unit in $\Z/4\Z$, so that the classical 
Witt cancellation lemma is not true. As usual, for any invertible 
element  $u$ of $\Z/4\Z$ we denote by $\langle u \rangle $ the 
non-degenerate symmetric bilinear form on $\Z/4\Z$  of determinant $u$.

\begin{proposition}\label{prop wittz4}
The Witt ring $W(\Z/4\Z)$ is isomorphic to $\Z/8\Z$, and it is generated 
by the class of $\langle -1\rangle$.
\end{proposition}
\noindent 
The computation of $W(\Z/4\Z)$ was 
obtained independently by Gurevich and Hadani in \cite{GH}.

\vspace{0.2cm}\noindent 
\vspace{0.2cm}\noindent  
The discriminant of $\omega=\left(
 \begin{matrix}
  2 & 1 \\
  1 & 2 
  \end{matrix}
\right)$ is $-1$ and in the proof of Proposition \ref{prop wittz4} below 
we show that its class is non-trivial in $W(\Z/4\Z)$ and  
hence it represents a non-trivial element in  
the kernel of the discriminant map 
$I(\Z/4\Z)\to (\Z/4\Z)^*/(\Z/4\Z)^{*2}$.
From the Cartesian diagram above we get that it also represents a non-trivial element in   the kernel of the leftmost vertical homomorphism 
$V(\Z/4\Z)\to K_1(\Z/4\Z)$. In particular  $ KSp_2(\Z/4\Z)$ is $\Z/2\Z$.  
%




\begin{proof}[Proof of Proposition \ref{prop wittz4}] 
According to \cite[I.3.3]{MH} every symmetric bilinear form 
is equivalent to a  direct sum of a diagonal form (with 
invertible entries) and a bilinear form $\beta$ on a submodule $N\subseteq V$ 
such that $\beta(x,x)$ is not a unit, for every $x\in N$. 
Thus, given a free $\Z/4\Z$-module $L$,  any non-degenerate 
symmetric bilinear form on  $L$ 
is an orthogonal sum of copies of $\langle 1 \rangle $, of 
$\langle -1 \rangle$ and of a bilinear 
form $\beta$ on a free summand $N$ such that for all 
$x \in N$ we have $\beta(x,x)\in\{0,2\}$. 
Fix a basis $e_1, \cdots, e_n$ of $N$. Let $B$ denote the matrix of $\beta$ in 
this basis. Expanding the determinant of $\beta$ along the first column we see 
that there must be an index $i \geq 2$ such that $\beta(e_1,e_i) =\pm1$, for 
otherwise the determinant would not be invertible. Without loss of generality 
we may assume that $i=2$ and that $\beta(e_1,e_2)=1$. Replacing if necessary 
$e_j$ for $j \geq 3$ by $e_j - \frac{\beta(e_1,e_j)} {\beta(e_1,e_2)}e_2$, we 
may assume that $B$ is of the form:
\[
\left(
\begin{matrix}
 s & 1& 0 \\
 1 & t & c\\
 0 & {c}^{\top}& A
\end{matrix}
\right)
\]
where $A$ and $c$ are a square matrix and a row matrix  respectively, 
of the appropriate sizes, and $s,t \in \left\{ 0,2\right\}$. Since  $st=0$,  the form $\beta$ 
restricted to the submodule generated by $e_1$ and $e_2$ defines a non-singular 
symmetric bilinear form and therefore 
$N = \langle e_1,e_2 \rangle \oplus \langle e_1,e_2 \rangle^\bot$, where on the first summand the bilinear form is either split (if  at least one of $s$ or $t$ is $0$) or is $\omega$.  
By induction we have that any symmetric bilinear form is an orthogonal
sum of copies of $\langle 1 \rangle, \langle -1 \rangle$, of 
\[
\omega =\left(
\begin{matrix}
 2 & 1 \\
 1 & 2 
 \end{matrix}
\right)
\]
and split spaces.

\vspace{0.2cm}
\noindent
It's a classical fact (see \cite[Chapter I]{MH}) that in $W(\Z/4\Z)$ 
one has $\langle 1 \rangle = - \langle -1 \rangle$.
Also $ \langle -1 \rangle \oplus \langle -1 \rangle\oplus 
\langle -1 \rangle\oplus \langle -1 \rangle$ is isometric to 
$-\omega \oplus \langle -1 \rangle\oplus \langle 1 \rangle$. 
To see this notice that,  if $e_1,\dots, e_4$ 
denotes the preferred basis for the former bilinear form, then  the matrix 
in the basis $e_1+e_2, e_1+e_3, e_1-e_2-e_3,e_4$ is 
precisely $-\omega \oplus \langle -1 \rangle \oplus\langle 1 \rangle$. Also, in the Witt ring $\langle 1 \rangle \oplus \langle -1 \rangle =0$, 
so $4 \langle -1 \rangle =-\omega$. Finally, if we denote by $e_1,e_2,e_3,e_4$ the preferred basis of the orthogonal sum $\omega \oplus \omega$, then a direct computation shows that the subspace generated by $e_1+e_3$ and $e_2+e_4$ is isotropic, hence $2\omega =0$ and in particular $\omega$ has order at most $2$. 
All these show that $W(\Z/4\Z)$ is generated by 
$\langle -1 \rangle $ and that this form is of order at 
most $8$. It remains to show that $\omega$ is a non-trivial 
element to finish the proof. 

\vspace{0.2cm}
\noindent
Assume the contrary, namely that there is a 
split form $\sigma$ such that $\omega \oplus \sigma$ 
is split. We denote by $A$ the underlying module of $\omega$ and by  $\{a,b\}$ 
its preferred basis. Similarly, we denote by $S$ the 
underlying space of $\sigma$ of dimension $2n$  and  
by $\{e_1, \dots,e_n, f_1, \dots,f_n\}$ a basis that exhibits 
it as a split form. Let $F$ denote the submodule generated by $ f_1, \dots,f_n$. 
By construction $e_1, \dots, e_n$ generate a totally isotropic submodule $E$ 
of rank $n$ in $A \oplus S$, and since it is included into a maximal 
isotropic submodule, we can adjoin to it a new element $v$ such that 
$v, e_1, \dots ,e_n$ is a totally isotropic submodule of $A \oplus S$, 
and hence has  rank $n+1$. By definition there are unique elements 
$x,y \in \Z/4\Z$ and elements $\varepsilon \in E$ and $\phi \in F$ such 
that $v = xa +yb + \varepsilon + \phi$. Since $v$ is isotropic we have: 
\[
 2x^2+ 2xy +2y^2 + 2\sigma(\varepsilon, \phi) +\sigma(\phi, \phi) \equiv 0 \; (\text{mod } 4).
\]
Since $E\oplus  \Z/4\Z v$ is totally isotropic, then  
$(\omega\oplus\sigma)(e_i,v) = 0$, for every 
$1 \leq i \leq n$. Now, $\sigma(e_i,\varepsilon)=0$, since $E$ is isotropic in $S$,  
$(\omega\oplus\sigma)(e_i,a)=0$ and  $(\omega\oplus\sigma)(e_i,b)=0$, so that 
$\sigma(e_i,\phi)=0$. 
In particular, 
since $\phi$ belongs to the dual module to $E$ with respect to $\sigma$, 
$\phi=0$, so the above equation implies:             
\[
 2x^2+ 2xy +2y^2  \equiv 0 \; (\text{mod } 4).
\]
But now this can only happen when $x$ and $y$ are multiples of $2$ in $\Z/4\Z$. 
Therefore reducing mod $2$, we find that $v \text{ mod } 2$ belongs to the 
$\Z/2\Z$ -vector space generated by the mod $2$ reduction of the elements 
$e_1, \dots, e_n$, and by Nakayama's lemma this contradicts the fact that 
the $\Z/4\Z$-module generated by $v,e_1,\dots, e_n$ has rank $n+1$. 
\end{proof}

\begin{remark}
Dennis and Stein proved in \cite{DS2} that $K_2(\Z/2^k\Z)=\Z/2\Z$ 
for any $k\geq 2$. If we had proved directly 
that the class of the symmetric bilinear form  
$\left(\begin{array}{cc} 
2^{k-1} & 1 \\
1 & 2^{k-1} \\
\end{array}\right)$ generates the kernel of 
the homomorphism $I(\Z/2^k\Z)\to (\Z/2^k\Z)^*/(\Z/2^k\Z)^{*2}$, which is 
of order two  for all $k\geq 2$, then the   Sharpe-type 
exact sequence of Barge and Lannes would yield $KSp_2(\Z/2^k\Z)=\Z/2\Z$, for 
any $k\geq 2$. This would permit us to do without Stein's stability results. 
However the description of the Witt group $W^f(\Z/2^k\Z)$ seems 
more involved for $k\geq 3$ and it seems more cumbersome than  
worthy to fill in all the details.  
\end{remark}

\subsection{Detecting the non-trivial class via Weil representations}\label{weilrep}
\subsubsection{Preliminaries on Weil representations}
The projective representation that we use is related to the theory of theta functions on 
symplectic groups, this relation being briefly explained in an appendix to this article.
Although the Weil representations 
of symplectic groups over finite fields of characteristic different from 2 
is a classical subject present in many textbooks, 
the slightly more general Weil representations associated to 
finite rings of the form $\Z/k\Z$ received less consideration until recently. 
They first appeared in print  as 
representations associated to finite abelian groups in 
\cite{Klo} for genus $g=1$  and were extended to 
locally compact abelian groups in \cite[Chapter I]{Weil}, 
and independently in the work of Igusa and Shimura  on theta functions 
\cite{Han,Igu,Shi} and in the physics literature \cite{HB}. They were rediscovered as monodromies of generalized  
theta functions arising in the $U(1)$ Chern-Simons theory 
in \cite{F1,F2,Go} and then in finite-time 
frequency analysis, see \cite{KN} and references 
from there. 
In \cite{F1,F2,Go} these are projective representations of the symplectic group 
factorizing through the  finite congruence quotients 
$Sp(2g, \Z/2k\Z)$, which are only defined for even $k\geq 2$.  
However, for odd $k$ the monodromy of theta functions leads to 
representations of the theta subgroup  of $Sp(2g, \Z)$.  
These also factor through the image of the theta group 
into the finite congruence quotients $Sp(2g, \Z/2k\Z)$. 
Notice however that the original Weil construction works as well for 
$\Z/k\Z$ with odd $k$, see e.g. \cite{GHH,KN}.

\vspace{0.2cm}
\noindent 
 It is well-known, see \cite[sections 43,44]{Weil} or \cite[Prop. 5.8]{RR},  
 that these projective  Weil representations lift 
 to linear representations of the integral metaplectic group, 
 which is the pull-back of the symplectic group in a double 
 cover of $Sp(2g,\R)$. The usual 
 way to resolve the projective ambiguities is to use the 
 Maslov cocycle (see e.g. \cite{Tu}). Moreover, it is known that 
 the Weil representations over finite fields of odd characteristic 
 and over $\C$ actually are linear representations. In fact 
 the vanishing of the second power 
 of the augmentation ideal of the Witt ring of such fields (see e.g. 
 \cite{Sus,Lam}) implies that the corresponding metaplectic extension 
 splits. This contrasts  with the fact 
 that Weil representations over $\R$ (or any local field different from 
 $\C$) are true representations of the real metaplectic group and cannot 
 be linearized (see e.g. \cite{Lam}). The Weil representations over local 
 fields of characteristic 2 are subtler as they are rather representations 
 of a double cover of the so-called pseudo-symplectic group 
 (see \cite{Weil} and \cite{GH} for recent work).

\vspace{0.2cm}
\noindent 
Let $k \geq 2$ be an integer, and denote by  
$\langle , \rangle$ the standard bilinear form on 
$(\Z/k\Z)^g \times (\Z/k\Z)^g\to  \Z/k\Z$. The Weil
 representation we consider is a representation in the 
unitary group of the 
complex vector space $\C^{(\Z/k\Z)^g}$ endowed with its 
standard Hermitian 
form. Notice that the canonical basis of this vector space 
is canonically labeled by elements in $(\Z/k\Z)^g$.

\vspace{0.2cm}
\noindent 
It is well-known (see e.g. \cite{Igusa}) that $Sp(2g,\Z)$ is generated by the matrices having one
of the following forms:
$
\left ( \begin{array}{cc}
    \mathbf 1_g & B \\
    0 & \mathbf 1_g
\end{array}  \right)$
where $B=B^{\top}$ has integer entries,
$
\left ( \begin{array}{cc}
    A & 0       \\
    0 & (A^{\top})^{-1}
\end{array}  \right)$
where $A\in GL(g,\Z)$ and  
$
\left ( \begin{array}{cc}
    0 & -\mathbf 1_g \\
    \mathbf 1_g & 0
\end{array}  \right)$. 

\vspace{0.2cm}\noindent
We can now define the Weil representations 
(see the Appendix for more details) 
on these generating matrices as follows:
\begin{equation}
\rho_{g,k}
\left ( \begin{array}{cc}
     \mathbf 1_g & B\\
     0 & \mathbf 1_g
\end{array}  \right)
= {\rm diag}\left(\exp\left(\frac{\pi \sqrt{-1}}{k}\langle m,Bm\rangle\right)\right)_{ m\in(\Z/k\Z)^{g}},
\end{equation}
where ${\rm diag}$ stands for diagonal matrix with given entries;

\begin{equation}
\rho_{g,k}
\left ( \begin{array}{cc}
   A & 0\\
   0 & (A^{\top})^{-1}
\end{array}  \right)
= (\delta_{A^{\top}m,n})_{m,n\in(\Z/k\Z)^{g}},
\end{equation}
where $\delta$ stands for the Kronecker symbol;
\begin{equation}
\rho_{g,k}
\left ( \begin{array}{cc}
    0 & -\mathbf 1_g\\
   \mathbf  1_g & 0
\end{array}  \right)
=k^{-g/2}\exp\left(-\frac{2\pi\sqrt{-1}\langle m,n\rangle}{k}\right)_{m,n\in({\Z}/k{\Z})^{g}}.
\end{equation}

\vspace{0.2cm}\noindent
It is proved in \cite{F2,Go}, that for even $k$ these formulas define a unitary  
representation $\rho_{g,k}$ of $Sp(2g,\Z)$ in 
$U(\C^{(\Z/k\Z)^g})/R_8$.  Here $U(\C^N)=U(N)$ denotes the unitary group 
of dimension $N$ and $R_8\subset U(1)\subset U(\C^N)$ is the subgroup of scalar 
matrices whose entries are roots of unity of order $8$.  
For odd $k$ the same formulas define representations of 
the theta subgroup $Sp(2g, 1,2)$ (see \cite{Igusa,Igu,F2}). 
Notice that by construction $\rho_{g,k}$ factors through 
$Sp(2g, \Z/2k\Z)$ for even 
$k$ and through the image of the theta subgroup in $Sp(2g,\Z/k\Z)$ 
for odd $k$. 

\vspace{0.2cm}\noindent
Our definition gives us a map 
$\rho_{g,k}:Sp(2g,\Z)\to U(\C^{(\Z/k\Z)^g})$ 
satisfying the cocycle condition: 
\[ \rho_{g,k}(AB)=\eta(A,B) \rho_{g,k}(A) \rho_{g,k}(B)\]
for all $A,B\in Sp(2g, \Z)$ and some 
$\eta(A,B)\in R_8$.

\subsubsection{Outline of the proof of Theorem \ref{tors-sympl0}}

The key step is to establish the following: 

\begin{proposition}\label{trivialmeta}
The projective Weil  representation $\rho_{g,k}$ of 
$Sp(2g,\Z)$, for $g\geq 3$ and  even $k$ does not lift to 
linear representations of $Sp(2g,\Z)$, namely it determines  
a generator of the group $H^2(Sp(2g,\Z/2k\Z); \Z/2\Z)$. 
\end{proposition}

\begin{remark}
For odd $k$ it was already known that Weil representations 
did not detect any non-trivial element, i.e. that the projective 
representation $\rho_{g,k}$  lifts to a linear representation \cite{AR}. We will
give a very short outline of this at the end of  the  Appendix.
\end{remark}

\vspace{0.2cm}\noindent Proposition \ref{alternative} states that the 
Schur multiplier $H_2(Sp(2g,\Z/D\Z))$ is either trivial or $\Z/2\Z$, while 
Proposition \ref{trivialmeta} provides an explicit nontrivial central  extension 
of   $Sp(2g,\Z/D\Z)$ by $\Z/2\Z$, when $D\equiv 0\,({\rm mod} \; 4)$. 
Therefore  $H_2(Sp(2g,\Z/D\Z))$ is nontrivial, and hence isomorphic to 
$\Z/2\Z$, thereby proving Theorem \ref{tors-sympl0}. 

\vspace{0.2cm}\noindent 
To prove Proposition \ref{trivialmeta} we first note that the projective  Weil representation $\rho_{g,k}$ determines a central 
extension of $Sp(2g, \Z/2k\Z)$ by $\Z/2\Z$, since it factors through the 
integral metaplectic group, by \cite{Weil}. 
We will prove that this central extension is non-trivial thereby proving the claim.  
The pull-back of this central extension along the homomorphism 
$Sp(2g,\Z)\to Sp(2g,\Z/2k\Z)$ is a central extension 
of $Sp(2g,\Z)$ by $\Z/2\Z$ and  it is enough to prove that this 
last extension is non-trivial. 

\vspace{0.2cm}\noindent 
It turns out to be easier to describe 
the pull-back of this central extension over the mapping class group  
$M_g$ of the 
genus $g$ closed orientable surface. 
\begin{definition}\label{theta}
Let $\widetilde{M}_g$ be the pull-back of the above
central extension,  associated to the projective Weil representation $\rho_{g,k}$, along the homomorphism 
$M_g\to Sp(2g,\Z)$. 
\end{definition}

\vspace{0.2cm}\noindent 
By the stability results of 
Harer (see \cite{Harer}) for $g \geq 5$, and the low dimensional computations in \cite{P} 
and \cite{KS} for $g \geq 4$,  the natural homomorphism $M_g\to Sp(2g,\Z)$, obtained 
by choosing a symplectic  basis in the surface homology, induces   
isomorphisms $H_2(M_g; \Z)\to H_2(Sp(2g,\Z); \Z)$ and 
$H^2(Sp(2g,\Z); \Z)\to H^2(M_g; \Z)$ for $g \geq 4$.  In particular in this
 range  the class of the central extension $\widetilde{M}_g$ is a generator 
of $ H^2(M_g;\Z/2\Z)$. 

\vspace{0.2cm}\noindent 
In contrast for $g = 3$, there is an element of 
infinite order in $H^2(M_3; \Z)$ such that its reduction mod $2$ is the 
class of the central extension $\widetilde{M}_3$, however   
$H^2(M_3; \Z/2\Z) = \Z/2\Z \oplus \Z/2\Z$. This follows from the 
computation $H_2(M_3; \Z)=\Z\oplus \Z/2\Z$, see e.g. \cite[Thm. 4.9, Cor. 4.10]{Sak}. Note that 
$H_2(Sp(6,\Z); \Z)=\Z\oplus \Z/2\Z$, according to Stein \cite{Stein3}. 

\vspace{0.2cm}\noindent 
Therefore, we can reformulate Proposition~\ref{trivialmeta} at least for $ g \geq 4$  
 in equivalent  form in terms of the 
mapping class group:

\begin{proposition}\label{classmcg}
If $g\geq 4$ then the class of the central extension $\widetilde{M}_g$ is 
a generator  of the group  
$H^2(M_g; \Z/2\Z) \simeq \Z/2\Z$. 
\end{proposition}

\vspace{0.2cm}\noindent 
The proof is an explicit computation, which turns out to be  also valid when $g=3$. This proves that  
the central extension coming from the Weil representation 
(and hence its cohomology class with $\Z/2\Z$ coefficients) 
for $g=3$ is non-trivial. This implies that $H_2(Sp(6,\Z/D\Z))\neq 0$, when  
$D\equiv 0\,({\rm mod} \; 4)$. In fact in \cite{BCRR} the authors proved 
that  $H_2(Sp(6,\Z/4\Z); \Z)=\Z/2\Z\oplus \Z/2\Z$. 

\vspace{0.2cm}\noindent 
The rest of this section is devoted to the proof of Proposition \ref{trivialmeta}. 

\subsubsection{A presentation of  $\widetilde{M}_g$ and the proof of Proposition  \ref{trivialmeta}}
The method we use is due to Gervais (see \cite{Ge}) and was already used 
in \cite{FK} for computing central extensions arising in quantum 
Teichm\"uller space. 
We start with a number of notations and definitions. Recall that  
$\Sigma_{g,r}$ denotes the orientable surface of genus $g$ with $r$ boundary components. If $\gamma$ is a curve on a surface, then $D_{\gamma}$ denotes the 
right Dehn twist along the curve $\gamma$.  

\begin{definition}
A chain relation $C$ on the surface $\Sigma_{g,r}$
is given by an embedding $\Sigma_{1,2}\subseteq \Sigma_{g,r}$
and the  standard chain relation on this 2-holed torus, namely
\[ (D_aD_bD_c)^4=D_eD_d,\]
where $a,b,c,d,e$ are the following curves of the embedded 2-holed torus:

\vspace{0.2cm}
\begin{center}
\includegraphics[scale=0.4]{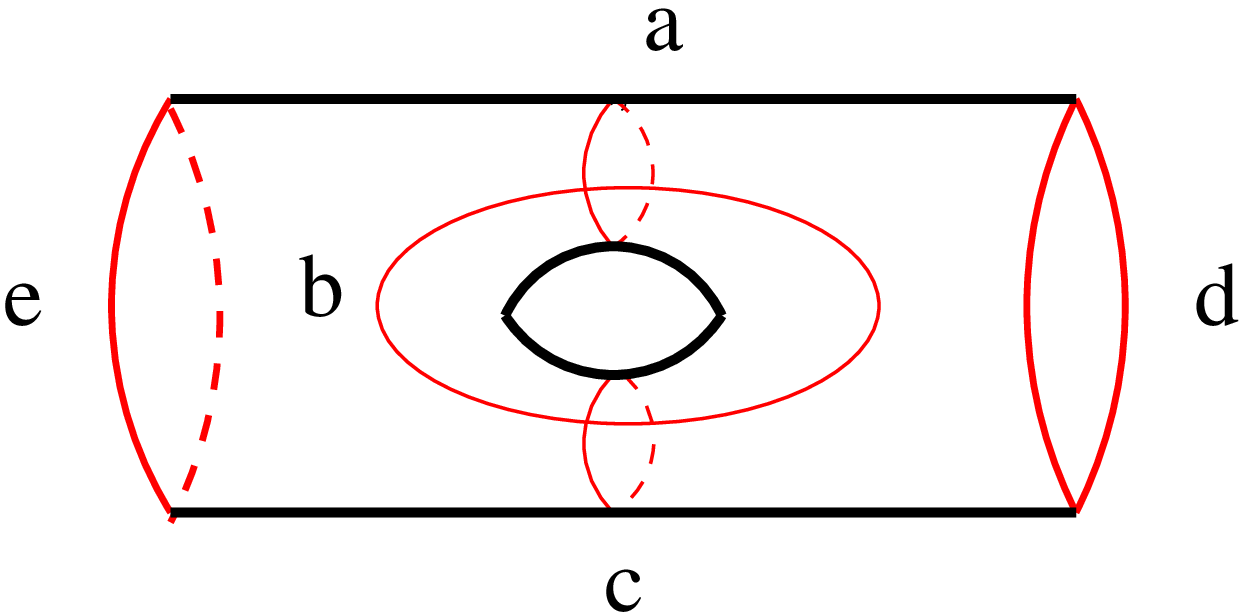}
\end{center}
\end{definition} 

\begin{definition}
A lantern relation $L$ on the surface $\Sigma_{g,r}$
is given by an embedding $\Sigma_{0,4}\subseteq \Sigma_{g,r}$
and the  standard lantern relation on this 4-holed sphere, namely
\begin{equation}D_{a_0}D_{a_1}D_{a_2}D_{a_3}= D_{a_{12}}D_{a_{13}}D_{a_{23}},
\end{equation}
where $a_0,a_1,a_2,a_3,a_{12},a_{13},a_{23}$
are the following curves of the embedded 4-holed sphere:

\vspace{0.2cm}
\begin{center}
\includegraphics[scale=0.4]{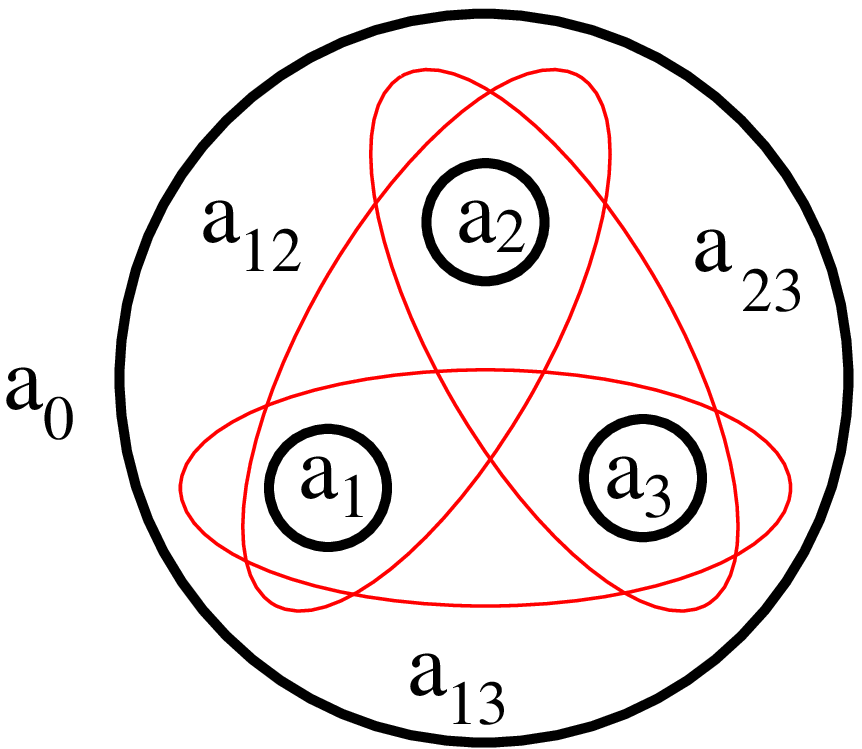}
\end{center}
\end{definition}

\vspace{0.1cm}\noindent 
The following lemma is a simple consequence of a deep result of Gervais~\cite[Thm. B]{Ge}:

\begin{lemma}\label{gervais}
Let $g\geq 3$, then the 
group $M_{g}$ has the following presentation:
\begin{enumerate}
\item Generators are all Dehn twists $D_{a}$ along all  non-separating
simple closed curves $a$ on $\Sigma_{g}$.
\item Relations:
\begin{enumerate}
\item Braid type 0 relations:
\[ D_aD_b=D_bD_a,\]
for each pair of disjoint non-separating simple closed curves $a$ and $b$;
\item   Braid type 1 relations:
\[ D_aD_bD_a=D_bD_aD_b,\]
for each pair of non-separating simple closed curves
$a$ and $b$ which intersect transversely in one point;
\item One lantern relation for a $4$-holed sphere embedded in $\Sigma_{g}$ so that all
boundary curves are non-separating;
\item One chain relation for a 2-holed torus embedded in $\Sigma_{g}$ so that all
boundary curves are non-separating.
\end{enumerate}
\end{enumerate}
\end{lemma}

\vspace{0.2cm}\noindent
The key step in proving Proposition \ref{classmcg} and hence  
Proposition \ref{trivialmeta} is to find an explicit presentation 
for the central extension $\widetilde{M}_{g}$ from  Definition \ref{theta}. 
If we choose arbitrary lifts $\widetilde{D}_a \in \widetilde{M}_g$  
for each of the Dehn twists $D_{a} \in M_g$, then $\widetilde{M}_g$ is 
generated by the elements $\widetilde{D}_a$ plus a central element 
$z$ of order  at most $2$.

\begin{proposition}\label{pres}
Suppose that $g\geq 3$. Then 
the group $\widetilde{M}_{g}$ has the following presentation.
\begin{enumerate}
\item Generators:
\begin{enumerate}
\item With each non-separating  simple closed curve $a$
in $\Sigma_{g}$ is associated a generator  $\widetilde{D}_{a}$;
\item One (central) element $z$, $z\neq 1$.
\end{enumerate}
\item Relations:
\begin{enumerate}
\item Centrality:
\begin{equation}
z \widetilde{D}_a=\widetilde{D}_az,
\end{equation}
for any non-separating  simple closed curve $a$
on $\Sigma_{g}$;
\item Braid type $0$ relations:
\begin{equation}\label{braid0rel}
\widetilde{D}_a\widetilde{D}_b=\widetilde{D}_b\widetilde{D}_a,\end{equation}
for each pair of disjoint non-separating simple closed curves $a$ and $b$;
\item   Braid type $1$ relations:
\begin{equation}\label{braid1rel} \widetilde{D}_a\widetilde{D}_b\widetilde{D}_a=\widetilde{D}_b\widetilde{D}_a
\widetilde{D}_b,
\end{equation}
for each pair of non-separating simple closed curves $a$ and $b$ which intersect transversely at one point;
\item One lantern relation for a $4$-holed sphere embedded in $\Sigma_{g}$ so that all
boundary curves are non-separating:
\begin{equation}\label{lanternrel}
 \widetilde{D}_{a_0}\widetilde{D}_{a_1}\widetilde{D}_{a_2}\widetilde{D}_{a_3}=
\widetilde{D}_{a_{12}}\widetilde{D}_{a_{13}}\widetilde{D}_{a_{23}},
\end{equation}
\item One chain relation for a 2-holed torus embedded in $\Sigma_{g}$ so that all
boundary curves are non-separating:
\begin{equation}\label{chainrel} (\widetilde{D}_a\widetilde{D}_b\widetilde{D}_c)^4=z
\widetilde{D}_e\widetilde{D}_d. \end{equation}
\item Scalar equation: 
\begin{equation}
z^2=1. 
\end{equation} 
\end{enumerate}
\end{enumerate}
\end{proposition}

\noindent 
Gervais \cite[Thm. C and Cor. 4.3]{Ge} proved that the universal central extension of the mapping class group
has the presentation given in Proposition \ref{pres} except for the relation (f) 
reading $z^2=1$. 
Therefore our group $\widetilde{M}_{g}$ from Definition~\ref{theta}
will be the non-trivial central extension of $M_{g}$ by $\Z/2\Z$ 
obtained from the universal central extension of $M_g$ by reducing mod $2$ 
its  kernel.  

\vspace{0.1cm}\noindent 
This will prove Proposition \ref{classmcg} and hence Proposition 
\ref{trivialmeta}.



\subsubsection{Proof of Proposition \ref{pres}}
By a slight abuse of language, we still denote  $\rho_{g,k}$ the  projective representation of 
$\M_g$ obtained from the Weil representation $\rho_{g,k}$ 
by the composition with the projection $M_g\to Sp(2g,\Z)$. 
Let then $\widetilde{\rho_{g,k}(M_g)}\subset  U(\C^{(\Z/k\Z)^g})$ be the pull-back of 
$\rho_{g,k}(M_g)\subset U(\C^{(\Z/k\Z)^g})/R_8$ along this composition.

\vspace{0.2cm}\noindent
By definition $\widetilde{M}_g$ fits into a commutative diagram:
\[
 \xymatrix{
0 \ar[r] & \Z/2\Z \ar@{=}[d] \ar[r] &  \widetilde{M}_g \ar[d] \ar[r] &  M_g \ar[d]^{\rho_{g,k}} \ar[r] &  1 \\
0 \ar[r] & \Z/2\Z \ar[r] &   \widetilde{\rho_{g,k}(M_g)}\ar[r] & \rho_{g,k}(M_g) \ar[r] &  1.
}
\]
This 
presents $\widetilde{M}_g$ as a pull-back and therefore the relations claimed 
in Proposition \ref{pres} will be satisfied if and only if they are satisfied 
when we project them both into $M_g$ and 
$\widetilde{\rho_{g,k}(M_g)} \subset U(\C^{(\Z/k\Z)^g})$. If this is the 
case then $\widetilde{M}_g$ will be a quotient of  the group obtained from the 
universal central extension by reducing mod $2$ the center and that 
surjects onto $M_g$. But, as the mapping class group is Hopfian there are  
only two  such groups: first, $M_g \times \Z/2\Z$ with the obvious 
projection on $M_g$ and second, the mod $2$ reduction of the 
universal central extension. Then relation $(e)$ 
shows that we are in the latter case. 

\noindent 
The projection on $M_g$ is obtained by killing the center $z$, and by 
construction the projected relations are satisfied in $M_g$ and we only need 
to check them in the unitary group.

\begin{lemma}
For any lifts  $\widetilde{D}_a$ of the Dehn twists $D_a$ we have
$\widetilde{D}_a\widetilde{D}_b=\widetilde{D}_b\widetilde{D}_a$ and
thus the braid type $0$ relations (b) are satisfied.
\end{lemma}
\begin{proof}
The commutativity relations are satisfied
for particular lifts and hence for arbitrary lifts. 
\end{proof}

\begin{lemma}\label{braid1}
There are lifts $\widetilde{D}_a$ of the Dehn twists $D_a$, for each
non-separating simple closed curve $a$  such that we have

\[\widetilde{D}_a\widetilde{D}_b\widetilde{D}_a=\widetilde{D}_b\widetilde{D}_a\widetilde{D}_b\]
for any simple closed curves $a,b$  with one intersection point and thus 
the braid type $1$ relations (c) are satisfied.
\end{lemma}
\begin{proof}
Consider an arbitrary lift of one braid type $1$ relation (to be called the fundamental one),
which has the form
$\widetilde{D}_a\widetilde{D}_b\widetilde{D}_a=z^k\widetilde{D}_b\widetilde{D}_a\widetilde{D}_b
$.   Change then  the lift $\widetilde{D}_b$ to $z^k\widetilde{D}_b$. With the new lift the
relation above becomes $\widetilde{D}_a\widetilde{D}_b\widetilde{D}_a=\widetilde{D}_b\widetilde{D}_a\widetilde{D}_b$.

\vspace{0.2cm}\noindent
Choose now an arbitrary braid type $1$ relation of $M_{g}$, say
$D_xD_yD_x=D_yD_xD_y$.   There exists a
1-holed torus $\Sigma_{1,1}\subset \Sigma_{g}$ containing $x,y$, namely a neighborhood of $x\cup y$.
Let $T$ be the similar torus containing $a,b$.  Since $a,b$ and $x,y$ are
non-separating there exists a
homeomorphism $\varphi:\Sigma_{g,r}\to \Sigma_{g,r}$ such that
$\varphi(a)=x$ and $\varphi(b)=y$. We have then
\[ D_x=\varphi D_a \varphi^{-1},\quad  D_y= \varphi D_b \varphi^{-1}.\]
Let us consider now an arbitrary lift $\widetilde{\varphi}\in\widetilde M_g$
of $\varphi$, which is well-defined only up to a central element, and set
\[ \widetilde{D}_x=\widetilde{\varphi} \widetilde{D}_a\widetilde{\varphi}^{-1},\quad
\widetilde{D}_y=\widetilde{\varphi} \widetilde{D}_b\widetilde{\varphi}^{-1}. \]
These lifts are well-defined since they do not depend on the choice of $\widetilde{\varphi}$ (the central elements  coming from $\widetilde{\varphi}$ and
$\widetilde{\varphi}^{-1}$ mutually cancel). Moreover, we have then
\[ \widetilde{D}_x\widetilde{D}_y\widetilde{D}_x=\widetilde{D}_y\widetilde{D}_x\widetilde{D}_y
\]
and so the braid type $1$ relations  (c) are all satisfied.
\end{proof}

\begin{lemma}
The choice of lifts of all $\widetilde{D}_x$, with $x$ non-separating,
satisfying the requirements of Lemma~\ref{braid1} is uniquely defined by
fixing the lift $\widetilde{D}_{a}$ of one particular Dehn twist.
\end{lemma}
\begin{proof}
In fact the choice of $\widetilde{D}_{a}$ fixes the choice of  $\widetilde{D}_{b}$.
If $x$ is a non-separating  simple closed curve on $\Sigma_{g}$, then
there exists another non-separating curve $y$ which intersects it in one point.
Thus, by Lemma~\ref{braid1}, the choice
of $\widetilde{D}_{x}$ is unique.
\end{proof}

\begin{lemma}
One can choose the lifts of Dehn twists in
$\widetilde{M}_{g}$ so that all braid type relations are satisfied and
the lift of the lantern relation is
trivial, namely
 \[ \widetilde{D}_a\widetilde{D}_b\widetilde{D}_c\widetilde{D}_d=\widetilde{D}_u
 \widetilde{D}_v\widetilde{D}_w,\]
for the non-separating curves in the fixed embedded $\Sigma_{0,4}\subset \Sigma_{g}$.
\end{lemma}
\begin{proof}
An arbitrary lift of that lantern relation is of the form
$\widetilde{D}_a\widetilde{D}_b\widetilde{D}_c\widetilde{D}_d
=z^k\widetilde{D}_u\widetilde{D}_v\widetilde{D}_w$. In this case, we change the lift $\widetilde{D}_a$ to $z^{-k}\widetilde{D}_a$ and adjust the lifts of all other Dehn twists along non-separating  curves in the unique way such that all braid type $1$-relations are satisfied. Then, the required form of the lantern relation is satisfied too.
\end{proof}

\vspace{0.2cm}\noindent
We say that the lifts of the Dehn twists are \emph{normalized} if all braid type
relations and one  lantern relation are lifted in a trivial way.

\vspace{0.2cm}\noindent
Now Proposition \ref{pres} follows from the following lemma, whose proof is rather calculatory 
and is postponed to the next subsection.  

\begin{lemma}\label{chainrelation}
If all lifts of the Dehn twist generators are normalized 
then $(\widetilde{D}_a\widetilde{D}_b\widetilde{D}_c)^4=z
\widetilde{D}_e\widetilde{D}_d$, where $z^2= 1$ and $z\neq 1$. 
\end{lemma}

\subsubsection{Proof of Lemma \ref{chainrelation}}
We denote by $T_{\gamma}$  the action of $D_{\gamma}$ in homology. 
Moreover we denote by $R_{\gamma}$ the matrix in $U(\C^{(\Z/k\Z)^g})$ 
corresponding to the  prescribed lift $\rho_{g,k}(T_{\gamma})$ 
of the projective representation.
The level $k$ is fixed through  this section and we drop 
the subscript $k$ from now on.   

\vspace{0.2cm}\noindent  
Our strategy is as follows. We show that the braid relations 
are satisfied by the matrices $R_{\gamma}$. It remains to compute the defect of the chain relation 
in the matrices $R_{\gamma}$. 

\vspace{0.2cm}\noindent 
Consider an embedding of $\Sigma_{1,2}\subset \Sigma_g$ such that 
all curves from the chain relation are non-separating, and thus 
like in the figure below:

\vspace{0.2cm}
\begin{center}
\includegraphics[scale=0.4]{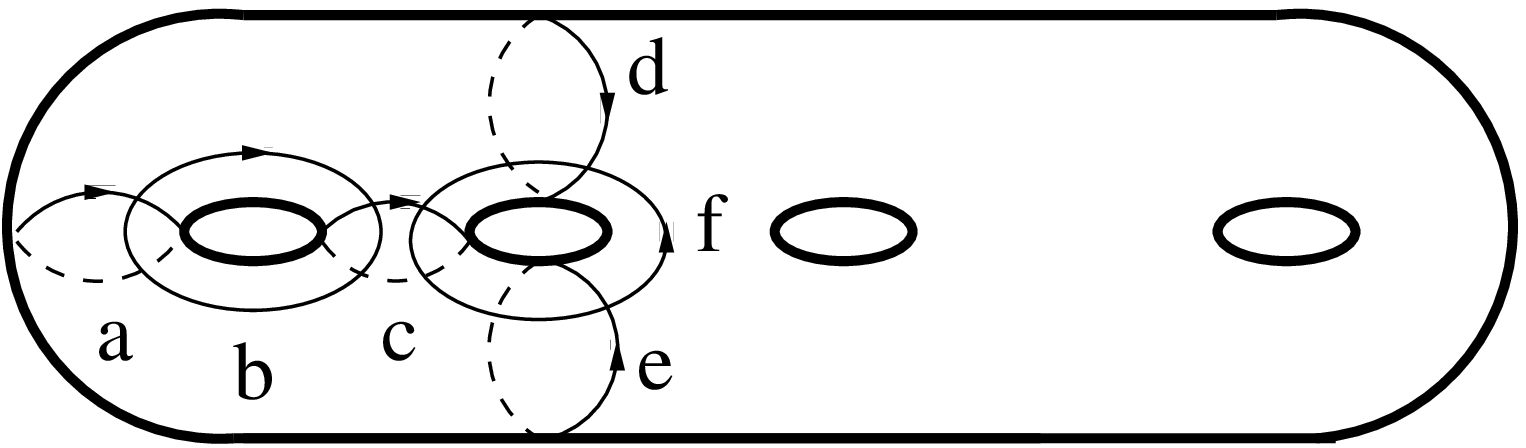}
\end{center}

\vspace{0.2cm}\noindent 
By construction, the action of the subgroup generated by $D_a,D_b,D_c,D_d,D_e$ and $D_f$ 
 on the homology of the surface $\Sigma_g$ preserves 
the symplectic subspace generated by the homology classes of 
$b, f, a, e$ and acts trivially on its orthogonal complement. 
Now the Weil representation behaves well 
with respect to the direct sum of symplectic matrices and this 
enables us to focus our attention on  the action of this subgroup 
on the 4-dimensional symplectic subspace generated by $b, f, a, e$ 
and to use  the representation $\rho_{2}$.  
In this basis the symplectic matrices associated to the above Dehn twists 
are: 
 
\[ T_a=\left(\begin{array}{cccc}
1 & 0 & 0 & 0 \\
0 & 1 & 0 & 0 \\
-1 & 0 & 1 & 0 \\
0 & 0 & 0 & 1\\
\end{array}
\right), \; 
 T_b=\left(\begin{array}{cccc}
1 & 0 & 1 & 0 \\
0 & 1 & 0 & 0 \\
0 & 0 & 1 & 0 \\
0 & 0 & 0 & 1\\
\end{array}
\right), \; 
T_c=\left(\begin{array}{cccc}
1 & 0 & 0 & 0 \\
0 & 1 & 0 & 0 \\
-1 & 1 & 1 & 0 \\
1 & -1 & 0 & 1\\
\end{array}
\right), \; 
\]
\[ T_d=T_e= \left(\begin{array}{cccc}
1 & 0 & 0 & 0 \\
0 & 1 & 0 & 0 \\
0 & 0 & 1 & 0 \\
0 & -1 & 0 & 1\\
\end{array}
\right), \; T_f= 
\left(\begin{array}{cccc}
1 & 0 & 0 & 0 \\
0 & 1 & 0 & 1 \\
0 & 0 & 1 & 0 \\
0 & 0 & 0 & 1\\
\end{array}
\right). \]
Notice that $T_b=J^{-1}T_aJ$, where $J$ is the matrix of the
standard symplectic structure. 

\vspace{0.2cm}\noindent 
Set $q=\exp\left(\frac{\pi i}{k}\right)$, which is a $2k$-th root of unity. 
We will change slightly  the basis $\{\theta_m, m\in(\Z/k\Z)^g\}$ 
of our representation vector space  
in order to exchange the two obvious parabolic subgroups of $Sp(2g,\Z)$. 
 
 \vspace{0.2cm}\noindent
The element 
$
 \rho_{g}
\left ( \begin{array}{cc}
    0 & -\mathbf 1_g\\
   \mathbf  1_g & 0
\end{array}  \right)
$
will be central in our argument, we will denote it by $S$.
Specifically we fix the basis given 
by $-S\theta_m$, with $m\in (\Z/k\Z)^g$.
We have then: 

\[ R_a = {\rm diag }(q^{\langle L_a x, x\rangle})_{x\in (\Z/k\Z)^2}, \; 
{\rm where} \;  
L_a =\left(\begin{array}{cc}
1 & 0  \\
0 & 0  \\
\end{array}
\right),\]
\[ R_c = {\rm diag }(q^{\langle L_c x, x\rangle})_{x\in (\Z/k\Z)^2}, \; 
{\rm where} \;  
L_c =\left(\begin{array}{cc}
1 & 1  \\
1 & 1  \\
\end{array}
\right),\]
and 
\[ R_e=R_d = {\rm diag }(q^{\langle L_e x, x\rangle})_{x\in (\Z/k\Z)^2}, \; 
{\rm where} \;  
L_e =\left(\begin{array}{cc}
0 & 0  \\
0 & 1  \\
\end{array}
\right).\]
We set now:  
\[ R_b=S^3R_aS \textrm{ and } R_f=S^3R_eS.\] 

\begin{lemma}\label{normalize}
The matrices $R_a,R_b,R_c,R_f,R_e$ are normalized lifts, namely 
the braid relations are satisfied. 
\end{lemma}
\vspace{0.2cm}\noindent
We postpone the proof of this lemma a few lines. 

\vspace{0.2cm}\noindent
For $u,v\in \mathbb N$, 
let us denote by $G(u,v)$ the Gauss sum: 
\[ G(u,v)=\sum_{x\in {\Z}/v{\Z}}\exp\left(\frac{2\pi \sqrt{-1}ux^2}{v}\right).\]
According to \cite[pp. 85-91]{Lang}  the value of the Gauss sum is 
\[ G(u,v) = d G\left(\frac{u}{d},\frac{v}{d}\right), \mbox{ if } {\rm g.c.d.}(u,v)=d,  \] 
and for  ${\rm g.c.d.}(u,v)=1$ we have:
\[ G(u,v)= \left\{
\begin{array}{ll}
\varepsilon(v)\left(\frac{u}{v}\right)\sqrt{v}, & \mbox{ for {\rm odd} } v, \\
0,                                              & \mbox{ for } v=2({\rm mod} ~ 4), \\
\overline{\varepsilon(u)}\left(\frac{v}{u}\right)\left(\frac{1+\sqrt{-1}}{\sqrt{2}}\right)\sqrt{2v}, & 
\mbox{ for } v=0({\rm mod} ~ 4),
\end{array}
\right. 
\]
where $\left (\frac{u}{v}\right)$ is the Jacobi symbol and 
\[ \varepsilon(a)= \left\{
\begin{array}{ll}
1, & \mbox{ if } a=1({\rm mod} ~ 4), \\
\sqrt{-1}, & \mbox{ if } a=3({\rm mod} ~ 4).
\end{array}
\right.
\]
Remember that the Jacobi (or the quadratic) symbol $\left(\frac{P}{Q}\right)$ 
is defined only for odd $Q$ by the  formula:
\[ \left(\frac{P}{Q}\right)=\prod_{i=1}^{s}\left(\frac{P}{q_i}\right) \]
where $Q=q_1q_2...q_s$ is the prime decomposition of $Q$, and for prime 
$q$ the quadratic symbol (also called the Legendre symbol in this setting) is 
given by: 
\[\left(\frac{P}{q}\right)=\left\{
\begin{array}{ll}
0, & \mbox{ if } P\equiv 0 ({\rm mod}~q)\\
1, & \mbox{ if } P=x^2({\rm mod}~q) \text{ and } P\not\equiv 0 ({\rm mod}~q),\\
-1, & \mbox{ otherwise}.
\end{array}
\right.
\]
while 
\[ \left(\frac{P}{1}\right)=1.\]

\vspace{0.2cm}\noindent 
The quadratic symbol satisfies the following reciprocity law 
\[ \left(\frac{P}{Q}\right)\left(\frac{Q}{P}\right)=(-1)^{\frac{P-1}{2}\frac{Q-1}{2}},\]
when both $P$ and $Q$ are odd. 

\vspace{0.2cm}\noindent 
Denote by $\omega = \frac{1}{2}G(1,2k)$. 
The lift of the chain relation is of the form: 
\[ (R_aR_bR_c)^4= \mu R_eR_d,\]
for some $\mu\in U(1)$. Our aim now is to compute the value of 
$\mu$. Set $X=R_aR_bR_c$, $Y=X^2$ and $Z=X^4$.
Let $m,n\in(\Z/k\Z)^2$, $m=\left(\begin{array}{c}
m_1 \\
m_2   \\
\end{array}
\right)$ and $n=\left(\begin{array}{c}
n_1 \\
n_2   \\
\end{array}
\right)$. 

The entry $X_{m,n}$ of the matrix $X$ is given by: 
\[ X_{m,n}=k^{-1}\omega \delta_{n_2,m_2}q^{-(n_1-m_1)^2+m_1^2+(n_1+n_2)^2}.\]
This implies $ Y_{m,n}=0$ if $\delta_{m_2,n_2}=0$. 
If $m_2=n_2$ then: 
\[ Y_{m,n}=
k^{-2}\omega^2 \sum_{r_1\in \Z/k\Z}q^{-(m_1-r_1)^2+m_1^2+(r_1+n_2)^2
-(n_1-r_1)^2+r_1^2+(n_1+n_2)^2}=\]
\[ = k^{-2}\omega^2 \sum_{r_1\in \Z/k\Z}q^{m_2^2+n_2^2+2n_1n_2+2r_1(m_1+m_2+n_1)}.\]
Therefore $Y_{m,n}=0$, unless $m_1+m_2+n_1=0$. Assume that 
 $m_1+m_2+n_1= 0$. Then: 
\[ Y_{m,n}= k^{-1}\omega^2 q^{m_2^2+n_2^2+2n_1n_2}=
  k^{-1}\omega^2 q^{-2m_1m_2}.\]
It follows that: 
$Z_{m,n}=\sum_{r\in (\Z/k\Z)^2} Y_{m,r}Y_{r,n}$
vanishes, except when $m_2=r_2=n_2$ and $r_1=-(m_1+m_2)$, 
$n_1=-(r_1+r_2)=m_1$. Thus $Z$ is a diagonal matrix. 
If $m=n$ then: 
\[ Z_{m,n}=   Y_{m,r}Y_{r,n}=
 k^{-2}\omega^4 q^{-2m_1m_2-2r_1r_2}= k^{-2}\omega^4 q^{2m_2^2}.\]
We have therefore obtained: 
\[ (R_aR_bR_c)^4=k^{-2}\omega^4 R_e^2\]
and thus $\mu=k^{-2}\omega^4=\frac{1}{k^2}\left(\frac{G(1,2k)}{2}\right)^4$. 
This proves that whenever $k$ is even we have $\mu=-1$. Since 
this computes the action of the central element $z$, it follows that $z \neq 1$. 
This ends the proof of Lemma \ref{chainrelation}.

\vspace{0.2cm}\noindent 
\begin{proof}[Proof of Lemma \ref{normalize}]
We know that  $R_b$ is  $S^3 R_a S$, where $S$ is 
the $S$-matrix, up to an eighth root of unity. 
The normalization of this root of unity is given by 
the braid relation: 
\[ R_a R_b R_a= R_b R_a R_b. \]
We have therefore: 
\[
(R_b)_{m,n}= k^{-2} \sum_{x\in (\Z/k\Z)^2} q^{\langle L_a x, x\rangle+ 2
\langle n-m, x\rangle}.\]
This entry vanishes except when $m_2=n_2$.  Assume  that $n_2=m_2$. Then: 
\[
(R_b)_{m,n}= k^{-1} \sum_{x_1\in \Z/k\Z} q^{x_1^2+ 2(n_1-m_1)x_1}=
k^{-1}q^{-(n_1-m_1)^2}\sum_{x_1\in \Z/k\Z} q^{(x_1+ n_1-m_1)^2}=
k^{-1}q^{-(n_1-m_1)^2} \omega\]
where $\omega=\frac{1}{2}\sum_{x\in \Z/2k\Z}q^{x^2}$ is 
a Gauss sum.
We have first: 
\[ (R_aR_bR_a)_{m,n}=k^{-1}\omega \delta_{m_2,n_2}q^{-(n_1-m_1)^2+m_1^2+n_1^2}=
k^{-1}\omega \delta_{m_2,n_2}q^{2n_1m_1}.\]
Further 
\[ (R_bR_a)_{m,n}=k^{-1}\omega \delta_{m_2,n_2}q^{-(n_1-m_1)^2+n_1^2}\]
so that: 
\[ (R_bR_aR_b)_{m,n}=k^{-2}\omega^2\sum_{r\in (\Z/k\Z)^2}\delta_{m_2,r_2}
 \delta_{n_2,r_2}q^{-(m_1-r_1)^2+r_1^2-(r_1-n_1)^2}= \]
\[=k^{-1}\omega^2 \delta_{m_2,n_2}q^{2m_1n_1}\sum_{r_1\in \Z/k\Z} 
q^{-(r_1-m_1+n_1)^2}=k^{-1}\omega \delta_{m_2,n_2}q^{2n_1m_1}.\]
Similar computations hold for the other pairs of non-commuting 
matrices in the set $R_b,R_c, R_f, R_e$.  
This ends the proof of Lemma \ref{normalize}.
\end{proof}

\section{Residual finiteness, finite quotients and their second homology}\label{residual}

\subsection{Residual finiteness for perfect groups}

Perfect groups have a universal central extension with 
kernel canonically isomorphic to their second integral 
homology group. In this section we show how to translate 
the residual finiteness problem for the universal central 
extension for a perfect group $\Gamma$ into a homological 
problem about $H_2(\Gamma)$. 
We will need the following lemmas from  \cite[ Lemma 2.1 \& Lemma 2.2]{FP}: 
\begin{lemma}\label{lift}
Let $\Gamma$ and $F$ be perfect groups, $\widetilde{\Gamma}$ and 
$\widetilde{F}$ their universal central extensions and 
$p: \Gamma \rightarrow F$ be a group homomorphism. Then there 
exists a unique homomorphism 
$\widetilde{p}:\widetilde{\Gamma}\to \widetilde{F}$ lifting $p$ such 
that the following diagram is commutative: 
\[ \begin{array}{clcrclc}
1 \to& H_2(\Gamma)&\to& \widetilde{\Gamma}&\to& \Gamma& \to 1 \\
     & p_*\downarrow & & \widetilde{p} \downarrow & & \downarrow p & \\
1 \to& H_2(F)&\to& \widetilde{F}&\to& F & \to 1  \\
\end{array}
\]
where $p_\ast= H_2(p)$ is the map induced by $p$ in homology.
\end{lemma}

\begin{lemma}\label{kernel}
Let $\Gamma$ be a finitely generated  perfect group and 
$\widetilde{\Gamma}$ be its universal central extension. 
We denote by $C$ the central group $\ker(\widetilde{\Gamma}\to\Gamma)$ of 
$\widetilde{\Gamma}$. 
\begin{enumerate}
\item Suppose that the finite index (normal) subgroup 
$H\subseteq \Gamma$ has the property that 
the image of $H_2(H)$ into $H_2(\Gamma)$ contains $dC$, for some $d\in\Z$. 
Let $F=\Gamma/H$ be the corresponding finite quotient 
of $\Gamma$ and $p:\Gamma\to F$ the quotient map. 
Then $d\cdot p_*(H_2(\Gamma))=0$, where $p_*:H_2(\Gamma)\to H_2(F)$ 
is the homomorphism induced by $p$. 
In particular, if $p_*:H_2(\Gamma)\to H_2(F)$ is surjective, then 
$d\cdot H_2(F)=0$. 
\item Assume that $F$ is a finite quotient of $\Gamma$ and let $d\in\Z$ such that   
$d\cdot p_*(H_2(\Gamma))=0$. For instance, this is satisfied when $d \cdot 
H_2(F)=0$. 
Let $\widetilde{F}$ denote the universal central extension of $F$. 
Then the homomorphism $p:\Gamma\to F$ has a unique lift  
$\widetilde{p}:\widetilde{\Gamma}\to \widetilde{F}$ and the kernel of 
$\widetilde{p}$  contains $d\cdot C$.  Observe that since $F$ is finite, $H_2(F)$ is also finite, hence we can take $d= |H_2(F)|$.
\end{enumerate} 
\end{lemma}

\subsection{Proof of Theorem \ref{boundtorsion0}}
Let $\mathbb K$ be a {\em number field}, ${\mathcal R}$ be the set of 
inequivalent valuations of $\mathbb K$ and 
$S\subset \mathcal R$ be a finite set of valuations of $\mathbb K$ 
including all the Archimedean (infinite) ones. 
Let 
\[ O(S)=\{x\in \mathbb K\colon\ v(x)\leq 1, \:\:  
{\rm for \:\: all } \: v\in \mathcal R\setminus S\}\]  
be the ring of $S$-integers in $\mathbb K$ and 
$\q\subset O(S)$ be a nonzero ideal. 
By $\mathbb K_v$ we denote the completion of $\mathbb K$ 
with respect to $v\in \mathcal R$. Following \cite{BMS}, we call
a domain $\A$ which arises as $O(S)$ above a 
{\em Dedekind domain of arithmetic type}.

\vspace{0.2cm}
\noindent
Let $\A=O(S)$ be a Dedekind domain of arithmetic type and 
$\q$ be an ideal of $\A$.  Denote by $Sp(2g,\A,\q)$ the kernel of the 
surjective homomorphism $p:Sp(2g,\A) \to Sp(2g, \A/\q)$. 
 The surjectivity is not a purely formal fact and  follows from 
the fact that in these cases the symplectic group coincides with 
the so-called "elementary symplectic group", and that it is trivial 
to  lift elementary generators of $Sp(2g, \A/\q)$ to $Sp(2g, \A)$;  
for a proof of this fact when $\A=\Z$ see \cite[Thm. 9.2.5]{HO}.
The restriction to $\mathbb K$ which are not totally imaginary 
comes from the  result of \cite{BMS} which states that symplectic groups 
$Sp(2g,\A)$, for $g\geq 3$ have the congruence subgroup property.

\vspace{0.2cm}
\noindent
Consider the central extension of $Sp(2g,\A)$  constructed by 
Deligne in \cite{De}, as follows. 
Let $\mu$ (respectively $\mu_v$ for a non-complex place $v$) 
be the group of roots of unity in $\mathbb K$ 
(and respectively $\mathbb K_v$). By convention one sets $\mu_v=1$ 
for a complex  place $v$.  
Moore showed in \cite{Moore} that there exists a universal topological 
central extension  $\widetilde{Sp(2g, \mathbb K_v)}$ of  
$Sp(2g, \mathbb K_v)$ by a discrete group $\pi_1(Sp(2g, \mathbb K_v))$.
When $v$ is a non real place Moore proved that there is an isomorphism 
between $\pi_1(Sp(2g, \mathbb K_v))$  and $\mu_v$. For a real 
place $v$, it is well-known that  $\pi_1(Sp(2g, \mathbb K_v))=\Z$.  
Set also $Sp(2g)_S=\prod_{v\in S}Sp(2g, \mathbb K_v)$ and recall that 
$Sp(2g,\A)$ is a subgroup of $Sp(2g)_S$. Then the universal 
topological central extension $\widetilde{Sp(2g)_S}$ of  
$Sp(2g)_S$ is isomorphic to the universal covering 
$\prod_{v\in S}\widetilde{Sp(2g, \mathbb K_v)}$ by the abelian group
$\pi_1(Sp(2g)_S)=\prod_{v\in S}\pi_1(Sp(2g, \mathbb K_v))$. 
Denote then by $\widetilde{Sp(2g,\A)}^D$ the inverse image of 
$Sp(2g,\A)$ in the universal covering  $\widetilde{Sp(2g)_S}$ of  
$Sp(2g)_S$. Then $\widetilde{Sp(2g,\A)}^D$ is a central extension 
of $Sp(2g,\A)$ which fits in an exact sequence: 

\begin{equation}
1\to \pi_1(Sp(2g)_S) \to \widetilde{Sp(2g,\A)}^D\to Sp(2g,\A) \to 1.
\end{equation}
There is a natural surjective homomorphism 
$\pi_1(Sp(2g, \mathbb K_v))\to \mu_v$, for all places $v$.  When composed 
with the map $\mu_v\to \mu$ sending $x$ to $x^{[\mu_v:\mu]}$ we 
obtain a homomorphism: 
\begin{equation}
R_S: \pi_1(Sp(2g)_S) \to \mu.
\end{equation}

\vspace{0.2cm}
\noindent
Deligne's theorem from \cite{De} states that the intersection of all finite index 
subgroups of $\widetilde{Sp(2g,\A)}^D$ coincides with $\ker R_S$, when $g\geq 2$. Then, 
Lemmas \ref{lift} and \ref{kernel} would prove the statement of Theorem 
\ref{boundtorsion0} if we knew that $\widetilde{Sp(2g,\A)}^D$ is the 
universal central extension of  $Sp(2g,\A)$. This is, for instance, the 
case when $\A=\Z$ and $g\geq 4$, but not true in full generality. 
In order to circumvent this difficulty, we drop out the torsion part of the kernels 
of the two central extensions.

\vspace{0.2cm}
\noindent
The key step in proving Theorem \ref{boundtorsion0} is a result stating 
the equivalence between the non-residual finiteness of Deligne's central extension 
and the existence of a uniform bound for the 2-homology of the 
finite congruence quotients, whose proof is postponed a few lines later.   

\begin{proposition}\label{symplgen}
The following statements are equivalent: 
\begin{enumerate}
\item There exists a homomorphism $R: \prod_{v\in S} \pi_1(Sp(2g, \mathbb K_v)) \to G$, where $G$ is a finite group such that every finite index subgroup of the Deligne central extension  
$\widetilde{Sp(2g,\A)}^D$, for $g\geq 3$, contains $\ker R$. 
\item  For fixed $\A$ and $g\geq 3$ 
there exists some uniform (independent on $g$ and $\q$) bound  for the 
size of the finite torsion groups $H_2(Sp(2g,\A/\q))$, 
for any  nontrivial ideal $\q$ of $\A$. 
\end{enumerate}
\end{proposition}

\begin{proof}[Proof of Theorem \ref{boundtorsion0}]
Deligne's result from \cite{De} yields an effective uniform bound for 
the size of the torsion group of $H_2(Sp(2g,\A/\q))$, 
since the first statement of Proposition \ref{symplgen} holds for  
$R=R_S$. Eventually, $Sp(2g,\A)$ has the congruence subgroup property, for $g\geq 2$, according to \cite{BMS}, namely any surjective homomorphism 
$Sp(2g,\A)\to F$ onto a finite group $F$ factors through some finite congruence 
quotient $Sp(2g,\A/\q)$. This proves our claim. 
\end{proof}

\vspace{0.2cm}
\noindent 
The main steps in the proof of Proposition \ref{symplgen} are the following.  
First, the natural homomorphism  $H_2(Sp(2g,\A)) \to H_2(Sp(2g,\A/\q))$ is 
surjective, for any $g\geq 3$.  Further, the   
groups $H_2(Sp(2g,\A))$ stabilize for large enough $g$ (depending on $\A$) 
so that there are upper bounds (independent on $g$) on the number of 
its generators and on the size of its torsion part. Therefore, the  finite abelian group 
$H_2(Sp(2g,\A/\q))$ has uniformy bounded size if and only if there is some 
uniform bound for the orders of the images of the generators of the 
free part of $H_2(Sp(2g,\A))$. The homomorphism 
between the universal central extension of $Sp(2g,\A)$ and the Deligne extension 
induces a map between their kernels 
$H_2(Sp(2g,\A))\to \prod_{v\in S} \pi_1(Sp(2g, \mathbb K_v))$.
This map is an isomorphism at the level of their free parts. 
Therefore $H_2(Sp(2g,\A/\q))$ has uniformly bounded size if and only if
the image of $\prod_{v\in S} \pi_1(Sp(2g, \mathbb K_v))$  through homomorphisms 
of the Deligne central extension into finite groups has uniformly bounded 
size.

\vspace{0.2cm}
\noindent 
Before proceeding, we  collect some of the results involved in the proof of Proposition \ref{symplgen}.  
We first need  the following technical proposition, whose proof can be found in \cite{FP}: 
\begin{proposition}\label{lem H2epi}
Given an ideal $\q \in \A$, for any $g \geq 3$, the homomorphism  $p_*:H_2(Sp(2g,\A)) \to H_2(Sp(2g,\A/\q))$ is 
surjective. 
\end{proposition}

\vspace{0.2cm}
\noindent 
Although the Deligne central extension is not in general the 
universal central extension, it is not far from it. 
The free part of the kernel of the universal central extension 
can be determined by means of:  

\begin{lemma}\label{BorSer}
We have: 
\[H_2(Sp(2g,\A); \R)= \R\otimes _{\Z}\prod_{v\in S\cap \mathbb R(S)} 
\pi_1(Sp(2g, \mathbb K_v)),\] 
where $\mathbb R(S)$ denotes the real 
Archimedean places in $S$. 
\end{lemma}
\begin{proof}
In the case where $\A$ is the ring of integers of a number field 
this is basically the result of Borel computing the stable 
cohomology of arithmetic groups from \cite[pp. 276]{Bo}. 
For the general case see \cite{BoSe,Bo2}. 
\end{proof}

\vspace{0.2cm}
\noindent 
Furthermore, we have the following general statement: 

\begin{lemma}\label{finitelygenerated}
The group $H_2(Sp(2g,\A))$ is finitely generated. 
\end{lemma}
\begin{proof}
This  follows from the existence of the Borel-Serre 
compactification \cite{BoSe} associated to an arithmetic group. 
\end{proof}
\vspace{0.2cm}
\noindent 
For a finitely generated abelian group $A$ we denote by $F(A)$ its free part and by $T(A)$ its torsion subgroup, so that $A=F(A)\oplus T(A)$. 

\vspace{0.2cm}
\noindent 
It is known that $Sp(2g,\A)$ is perfect, when $g\geq 3$. It has then a 
universal central extension $\widetilde{Sp(2g,\A)}$ by $H_2(Sp(2g,\A))$. 
Consider the quotients 
$E=\widetilde{Sp(2g,\A)}/T(H_2(Sp(2g,\A)))$ 
and $D=\widetilde{Sp(2g,\A)}^D/T(\pi_1(Sp(2g)_S))$. 
Then $E$ and $D$ are central extensions of $Sp(2g,\A)$ by torsion-free groups. 

\begin{lemma}\label{univDeligne}
There is a natural embedding of central extensions $E\to D$ 
which lifts the identity of $Sp(2g,\A)$ and identifies $E$ with a finite index normal subgroup of $D$. 
\end{lemma}
\begin{proof}
From Lemma \ref{BorSer} and the above description of fundamental groups  
$\pi_1(Sp(2g, \mathbb K_v))$ due to Moore \cite{Moore}, both $F(H_2(Sp(2g,\A)))$ and  $F(\pi_1(Sp(2g)_S))$ are isomorphic to the abelian group $\prod_{v\in S\cap \mathbb R(S)} \pi_1(Sp(2g, \mathbb K_v))$, because the non real places only provide torsion factors of $\pi_1(Sp(2g)_S)$. 

\vspace{0.2cm}
\noindent 
By universality of the central extension $\widetilde{Sp(2g,\A)}$, there is a homomorphism 
$\widetilde{Sp(2g,\A)}\to  \widetilde{Sp(2g,\A)}^D$ lifting the identity 
of $Sp(2g,\A)$; it  induces a homomorphism between the quotients  
$\iota:E\to D$ sending the kernel of the first extension into the kernel of the second one, namely  
such that $\iota(F(H_2(Sp(2g,\A))))\subseteq F(\pi_1(Sp(2g)_S))$. 
The linear map  induced between the associated real vector spaces 
$\iota\otimes 1_{\mathbb R}:H_2(Sp(2g,\A))\otimes_{\Z} \R\to \pi_1(Sp(2g)_S)\otimes_{\Z} \R$ 
can be identified with the map induced in homology by the 
inclusion $Sp(2g,\A)\to Sp(2g)_S$. Recall that $Sp(2g,\A)$ is a discrete subgroup of 
the locally compact group $Sp(2g)_S$. A consequence of the Garland-Matsushima vanishing theorem from \cite[Thm. 6.4]{Bo4} states that the map induced by the inclusion at the level of $H_2$ is an isomorphism for large enough $g$. This implies that the restriction 
$\iota:F(H_2(Sp(2g,\A)))\to F(\pi_1(Sp(2g)_S))$ is  injective. Thus $\iota$ should be injective, 
as well, because it is a lift of the identity of $Sp(2g,\A)$.   
\end{proof}

\begin{lemma}\label{bound}
The orders of the groups $T(H_2(Sp(2g,\A)))$, $g\geq 2$, are bounded from above by a constant which only depends on $\A$. 
\end{lemma}
\begin{proof}
From Stein's  surjective stability \cite{Stein2} the 
embeddings $Sp(2g,\A)\to Sp(2g+2,\A)$  induce surjective maps 
$H_2(Sp(2g,\A))\to H_2(Sp(2g+2,\A))$, for  $g$ larger than a constant depending on $\A$. 
The claim follows now, because the abelian groups $H_2(Sp(2g,\A))$ are finitely generated,
by Lemma \ref{finitelygenerated}.  
\end{proof}

\vspace{0.2cm}
\noindent 
\begin{proof}[Proof of Proposition \ref{symplgen}]
Recall that $p:Sp(2g,\A)\to Sp(2g,\A/q)$ is the surjective homomorphism 
induced by the reduction mod $q$. 
Let $p_*: H_2(Sp(2g,\A))\to H_2(Sp(2g,\A/q))$ be the corresponding 
homorphism in homology, which is surjective according to Proposition \ref{lem H2epi}. 
By Lemma \ref{lift} there exists a lift 
$\tilde p: \widetilde{Sp(2g,\A)}\to \widetilde{Sp(2g,\A/q)}$ of $p$, whose restriction 
to the kernel is precisely $p_*$. 

\vspace{0.2cm}
\noindent
Bounding  the size of $H_2(Sp(2g,\A/q))$  
is equivalent to bounding the size of  
the finite group $F=H_2(Sp(2g,\A/q))/p_*(T(H_2(Sp(2g,\A))))$,  according to Lemma \ref{bound}. 
Note that $\tilde p$ induces a  surjective homomorphism  
$\tilde p: E\to F$. 
 
\vspace{0.2cm}
\noindent 
Assume now that the Deligne central extension has the property from the 
first statement of the proposition. 
Lemma \ref{univDeligne} identifies $E$ 
with a  finite index subgroup of $D$. Consider the
representation ${\rm Ind}_E^D\tilde p:D\to \tilde F$ induced from $\tilde p: E\to F$. 
Its image is a finite group $\tilde F$, which is a quotient of the Deligne extension 
$\widetilde{Sp(2g,\A)}^D$. By our  hypothesis, we have $\ker {\rm Ind}_E^D\tilde p \supseteq \ker R$,  
and hence there is some surjective homomorphism $\lambda:G \to \tilde F$ so that: 
\[  {\rm Ind}_E^D\tilde p|_{\pi_1(Sp(2g)_S)}= \lambda \circ R.\]
In particular the image $\tilde{p}(H_2(Sp(2g,\A)))\subseteq \tilde F$ 
is covered by $G$ and hence has uniformly bounded size.

\vspace{0.2cm}
\noindent
Conversely, assume that there exists some $k(\A)\in \Z\setminus\{0\}$ such that 
 $k(\A)\cdot H_2(Sp(2g,\A/\q))=0$, for every ideal $\q \subset \A$.  Then 
the surjectivity of $p_*$ implies that 
$p_*(k(\A)\cdot c)=0$, for every ideal $\q$ and $c\in H_2(Sp(2g,\A))$. 

\vspace{0.2cm}
\noindent
This implies that $f_*(|\mu|\cdot k(\A)\cdot c)=0$, 
for every morphism $f:Sp(2g,\A)\to F$
onto a finite group $F$, where $|\mu|$ is the cardinal of the finite group $\mu$. 
In fact,  by  \cite{BMS} the congruence subgroup kernel is 
the finite cyclic group $\mu$, 
when $\mathbb K$ is totally imaginary (i.e. it has no non-complex 
places), and is trivial if $\mathbb K$ has 
at least one non-complex place. This means that for any 
finite index normal subgroup $H$ (e.g. $\ker f$) there exists 
an elementary subgroup $ESp(2g,\A,\q)$ contained in $H$, where 
$ESp(2g,\A,\q)\subseteq Sp(2g,\A,\q)$ is  a normal subgroup 
of finite index dividing $|\mu|$. Therefore $f(-)$ factors through 
 the quotient $Sp(2g,\A)/ESp(2g,\A,\q)$ and the composition 
$f({|\mu|}\cdot -)$ factors through $Sp(2g,\A/\q)$. 
Since $H_2(Sp(2g,\A/\q))$ is $k(\A)$-torsion 
we obtain $f_*(|\mu|\cdot k(\A)\cdot c)=0$, as claimed. 

\vspace{0.2cm}
\noindent
In particular, we can apply this equality to the morphism $f$ 
between the universal central extensions of $Sp(2g,\A)$ and $F$. 
The restriction of  $f_*$ to the free part 
of $H_2(Sp(2g,\A))$ is then trivial on multiples of 
$|\mu|\cdot k(\A)$. Thus these multiple elements lie in 
the kernel of any homomorphism of $\widetilde{Sp(2g,\A)}$ into a finite group. 
This proves Proposition \ref{symplgen}. 
\end{proof}

\subsection{Proof of Theorem \ref{nonresid0}}
Consider an arbitrary surjective homomorphism 
$\widehat{q}:\widetilde{Sp(2g,\Z)}\to \widehat{F}$ onto some finite 
group $\widehat{F}$. We set $F=\widehat{F}/\widehat{q}(C)$, where 
$C=\ker(\widetilde{Sp(2g,\Z)}\to Sp(2g,\Z))$. Then there is 
an induced homomorphism $q:Sp(2g,\Z)\to F$. Since $F$ is finite the 
congruence subgroup property for the symplectic groups 
(see \cite{BMS,Me1,Me2}) implies that there is some $D$ such that $q$ 
factors as $s\circ p$, where $p:Sp(2g,\Z)\to Sp(2g,\Z/D\Z)$ is the 
reduction mod $D$ and $s: Sp(2g,\Z/D\Z)\to F$ is a surjective homomorphism. 
\[
\xymatrix{
1 \ar[r] & C \ar[r] \ar[dd]& \widetilde{Sp(2g,\Z)} \ar[r] \ar[dd]^{\widehat{q}} &  Sp(2g,\Z) \ar[r] \ar[dd]^q \ar[dr]^p& 1 \\
 & & & & Sp(2g,\Z/D\Z) \ar[dl]^s \\
 1 \ar[r] & \widehat{q}(C) \ar[r] & \widehat{F} \ar[r] & F \ar[r] & 1
} 
\]
\vspace{0.2cm}\noindent 
Since $F$, a quotient of a symplectic group with $g\geq 4$  , is perfect it has a universal central extension $\widetilde{F}$. 
By Lemma \ref{lift}, there exists then unique lifts 
$\widetilde{p}: \widetilde{Sp(2g,\Z)}\to  \widetilde{Sp(2g,\Z/D\Z)}$, 
$\widetilde{q}: \widetilde{Sp(2g,\Z)}\to \widetilde{F}$ and 
$\widetilde{s}:\widetilde{Sp(2g,\Z/D\Z)}\to \widetilde{F}$ of 
the homomorphisms $p,q$ and $s$, respectively such that 
$\widetilde{q}=\widetilde{s}\circ 
\widetilde{p}$.
\[
\xymatrix{
1 \ar[r] & H_2F \ar[r] \ar[d] & \widetilde {F} \ar[r] \ar[d]^\theta & F \ar@{=}[d] \ar[r]& 1 \\
1 \ar[r] & \widehat{q}(C) \ar[r] & \widehat{F} \ar[r] & F \ar[r] & 1
}
\]
Since $\widetilde{F}$ is universal there is a unique  
homomorphism $\theta:\widetilde{F}\to \widehat{F}$ which lifts the identity 
of $F$.  We claim that $\theta\circ \widetilde{q}=\widehat{q}$. 
Both homomorphisms are lifts of $q$ and $\widetilde{q}=\widetilde{s}\circ\widetilde{p}$, by  the uniqueness claim of Lemma \ref{lift}. 
Thus  $\theta\circ \widetilde{q}=\alpha\cdot \widehat{q}$, where 
$\alpha$ is a homomorphism on 
$\widetilde{Sp(2g,\Z)}$ with target 
$\widehat{q}(C)=\ker(\widehat{F}\to F)$, which is central in $\widehat{F}$. 
Since $\widetilde{Sp(2g,\Z)}$ is universal we have 
$H_1(\widetilde{Sp(2g,\Z)})=0$ and thus $\alpha$ is trivial, as claimed. 
Summing up, we have 
$\widehat{q}=\theta\circ\widetilde{q}=\theta\circ \widetilde{s}\circ\widetilde{p}$.

\vspace{0.2cm}\noindent 
We know that $H_2(Sp(2g,\Z/D\Z))\in \{0,\Z/2\Z\}$, when $g\geq 4$,  and  from
Lemma \ref{lift} we obtain that  $\widetilde{p}(2c)=2\cdot p_*(c)=0$, 
where $c$ is the generator of $H_2(Sp(2g,\Z))$. In particular, 
$2\cdot \widehat{q}(c)=0\in \widehat{F}$, as claimed.  
 
\vspace{0.2cm}\noindent
Moreover, Lemma \ref{lift} along with  Proposition \ref{lem H2epi} provide a   
surjective homomorphism onto the universal central extension of $Sp(2g,\Z/D\Z)$, 
when $D\equiv 0\; ({\rm mod}\; 4)$. 
Thus, by Theorem  \ref{tors-sympl0} 
the image of the center of $\widetilde{Sp(2g,\Z)}$  has order two, as claimed.

\begin{remark}
Note that
 $H_2(Sp(6,\Z/D\Z))\in \{0,\Z/2\Z,\Z/2\Z\oplus \Z/2\Z\}$, extending \cite[Prop. 3.1]{FP}.
This is a consequence of Proposition  \ref{lem H2epi} and 
the fact that $H_2(Sp(6,\Z); \Z)=\Z\oplus \Z/2\Z$, according to \cite{Stein3}. 
Then Proposition  \ref{trivialmeta} shows that the image of some element of the center of $\widetilde{Sp(6,\Z)}$ in $\widetilde{Sp(6,\Z/D\Z)}$ has order two, when $D\equiv 0\; ({\rm mod}\; 4)$. 
\end{remark}

\begin{remark}
Using the results in \cite{BM1,BM2} we can obtain 
that $H_2(Sp(4,\Z/D\Z))=0$, if $D\neq 2$ is prime and 
 $H_2(Sp(4,\Z/2\Z))=\Z/2\Z$. 
Notice that $Sp(4,\Z/2\Z)$ is not perfect, but the extension 
$\widetilde{Sp(4,\Z)}$ still makes sense. 
\end{remark}

\vspace{0.2cm}\noindent 
An immediate corollary  of Theorem  \ref{tors-sympl0}  is the following 
$K$-theory result:

\begin{corollary}
For $g\geq 4$ we have 
\[KSp_{2,2g}(\Z/D\Z)=
\left\{\begin{array}{ll} 
\Z/2\Z,  & {\rm if }\; D\equiv 0({\rm mod} \; 4), \\
0, & {\rm otherwise}.
\end{array}\right.
\]  
\end{corollary}
\begin{proof}
According to Stein's stability results (\cite{Stein}, 
Thm 2.13, and \cite{Stein2}) we have 
\begin{equation} 
KSp_{2,2g}(\Z/D\Z)\cong KSp_{2}(\Z/D\Z), \: {\rm  for } \:\: g\geq 4,
\end{equation} 
and 
\begin{equation}
KSp_{2,2g}(\Z/D\Z)\cong H_2(Sp(2g,\Z/D\Z)), \: {\rm  for } \:\: g\geq 4.
\end{equation}
\end{proof}

\begin{remark}
The analogous result that  $K_{2,n}(\Z/D\Z)\in\{0,\Z/2\Z\}$, 
if $n\geq 3$ was proved a long time ago (see \cite{DS2}).
\end{remark}

\section{Mapping class group quotients}\label{mcg}

\subsection{Preliminaries on quantum representations}
The results of this section are the counterpart of those 
obtained in Section \ref{weilrep}, by considering $SU(2)$ 
instead of abelian quantum representations. 
The first author proved in \cite{Fu3} that central extensions of 
the  mapping class group $M_g$ by $\Z$ are residually finite. The 
same method actually  can be used to show the abundance of finite 
quotients with large torsion in their  second homology and to prove that the mapping class group has  property $A_2$  trivial modules. 

\vspace{0.2cm}\noindent 
A  quantum representation is a projective 
representation, depending on an integer $k$,  which lifts to a linear
representation $\widetilde{\rho}_k: M_g(12)\to U(N(k,g))$ of the 
central extension $M_g(12)$ of the mapping class group 
$M_g$ by $\Z$. 
The latter representation corresponds to invariants of 3-manifolds 
with a $p_1$-structure. Masbaum, Roberts (\cite{MR}) and 
Gervais (\cite{Ge}) gave a precise  description 
of this extension.  Namely, the 
cohomology class $c_{M_g(12)}\in H^2(M_g,\Z)$  associated 
to this extension equals 12 times the signature class $\chi$. It is 
known (see \cite{KS}) that the group $H^2(M_g)$ is generated by $\chi$, 
when $g\geq 3$. Recall that $\chi$ is one fourth  of the Meyer 
signature class. We denote more generally by $M_g(n)$ the central extension 
by $\Z$ whose class is $c_{M_g(n)}=n\chi$.

\vspace{0.2cm}\noindent 
It is known that $M_g$ is perfect and  $H_2(M_g)=\Z$, when $g\geq 4$ 
(see \cite{P}, for instance). Thus, for $g\geq 4$, 
$M_g$ has a universal central extension by $\Z$, 
which can be identified with the central extension 
$M_g(1)$. This central extension makes sense for 
$M_3$, as well, although it is not the universal central extension since 
$H_2(M_3)=\Z\oplus \Z/2\Z$. However, using its explicit presentation 
for $g=3$ we derive that $M_g(1)$ is perfect for $g\geq 3$.

\vspace{0.2cm}\noindent 
Let $c$ be the generator of the center of 
$M_g(1)$, which is 12 times the generator of the 
center of $M_g(12)$.  Denote by $M_g(1)_n$ the quotient of $M_g(1)$ 
obtained by imposing $c^n =1$; this is a non-trivial central extension 
of the mapping class group by $\Z/n\Z$. We will say that a quantum 
representation $\widetilde{\rho_p}$ \emph{detects}  the center of 
$M_g(1)_n$ if it factors through $M_g(1)_n$ and is injective on its center.

\vspace{0.2cm}\noindent 
Consider  the SO(3)-TQFT  with parameter  
$A=-\zeta_{p}^{(p+1)/2}$, where $\zeta_p$ is a primitive $p$-th root of unity, 
so that $A$ is a primitive $2p$ root of unity  with $A^2=\zeta_p$.  
This data leads to a quantum representation $\widetilde{\rho}_p$ for which  
Masbaum and Roberts computed in \cite{MR} that
$\widetilde{\rho}_p(c)=A^{-12-p(p+1)}$.

 \begin{lemma}\label{lem detect}
For each prime power $q^s$ there exists 
some quantum representation $\widetilde{\rho_p}$ 
which detects the center of $M_g(1)_{q^s}$.
\end{lemma}
\begin{proof}
We noted that $\widetilde{\rho}_p(c)=\zeta_{2p}^{-12- p(p+1)}$, 
where $\zeta_{2p}$ is a $2p$-root of unity. 
\begin{enumerate}
\item If  $q$ is a prime number $q\geq 5$ we let $p=q^s$. 
Then $2p$ divides $p(p+1)$ and  
$\widetilde{\rho}_p(c)=\zeta_{2p}^{-12} = \zeta_{p}^{-6}$ 
is of order $p = q^s$. Thus the representation $\widetilde{\rho}_p$
detects the  center of $M_g(1)_{q^s}$.
\item If $q=2$, we set  $p=2$. Then  
$\widetilde{\rho}_2(c)=\zeta_2$, and $\widetilde{\rho_2}$ detects the  center of
 $M_g(1)_2$.
\item Set now $p= 12r$ for some integer $r>1$ to be fixed later. 
Then $\widetilde{\rho}_p(c)=\zeta_{24r}^{-12- 12r(12r+1)}= \zeta_{2r}^{-1 -r(12r+1)} = \zeta_{2r}^{-1-r}$. 
This $2r$-th root of the unit has order ${\rm l.c.m.}(1+r,2r)/(1+r) = 2r/{\rm g.c.d.}(1+r,2r)$. An 
elementary computation shows that ${\rm g.c.d.}(1+r,2r) =1$ or 
$2$ depending on whether $r$ is even or odd. 
\begin{itemize}
\item If $q=2^{s+1}$ with $s \geq 1$, we set $r= 2^s$. Then $\zeta_{2\cdot 2^s}^{-1-2^s}$ is of order $2^{s+1}$, 
the representation  $\widetilde{\rho}_p$ detects the  center of  
 $M_g(1)_{2^{s+1}}$.   
\item If $g = 3^s$ with $s\geq 1$, we set  $r= 3^s$. Then $\zeta_{2\cdot 3^s}^{-1-3^s}$ is of order $3^{s}$, 
the representation  $\widetilde{\rho}_p$ detects the  center of 
 $M_g(1)_{3^{s}}$. 
\end{itemize}
\end{enumerate}
\end{proof}


\subsection{Proof of Theorem \ref{res0}} 
We have first the following  lemma from (\cite{FP}, Lemma 2.3): 

\begin{lemma}\label{key lemma}
Let $G$ be a perfect finitely presented group, 
$\widetilde{G}$ denote its universal central extension, 
$p:G\to F$ be a surjective homomorphism onto a finite group $F$ and 
$\widehat{p}:\widetilde{G}\to \Gamma$ be some lift of $p$ to 
a finite central extension $\Gamma$ of $F$. 
Assume that the image $C=\widehat{p}(Z(\widetilde{G}))$ 
of the center $Z(\widetilde{G})$ of $\widetilde{G}$ 
contains an element of order $q$. Then there exists an element 
of $p_*(H_2(G))\subseteq H_2(F)$ of order $q$. 
\end{lemma}

\begin{remark}\label{caution}
A cautionary remark is in order. Assume that $H_2(G)=\Z$ in the 
lemma above and that let $H = \ker p$. 
If the image of $H_2(H)$ in $H_2(G)$ is $d\Z$ then 
we can only assert that the image $p_*(H_2(G))$ in $H_2(F)$ 
is of the form $\Z/d'\Z$, for some divisor $d'$ of $d$, which might 
be proper divisor.  For instance taking $G=Sp(2g,\Z)$ and 
$F=Sp(2g,\Z/D\Z)$, where $D$ is even and not multiple of $4$, 
then $d=2$ by \cite[Thm. F]{Pu1}, while $d'=1$ as 
$H_2(Sp(2g,\Z/D\Z))=0$. The apparent contradiction with 
Lemma \ref{key lemma} comes from the fact that we required the existence of 
a lift $\widehat{p}: \widetilde{G}\to \Gamma$ for which the image of the center has order $d$. 
 If $H$ were perfect, as it is the case 
when $F=Sp(2g,\Z/2\Z)$ then the 
universal central extension $\widetilde{H}$ would come with a 
homomorphism $\widetilde{H}\to \widetilde{G}$. Although the image 
of $\widetilde{H}$ is a finite index subgroup, it is not, in general, a normal 
subgroup of $\widetilde{G}$. Passing to a finite index 
normal subgroup of $\widetilde{H}$ would amount to change 
$H$ into a smaller subgroup $H'$ and hence $F$ is replaced by a larger  
quotient $F'$.  
\end{remark}

\noindent 
\begin{proof}[Proof of Theorem \ref{res0}] By a classical 
result of Malcev \cite{Mal40}, finitely generated subgroups of linear 
groups over a commutative unital ring are residually finite. This applies 
to the images of quantum representations. Hence there are finite quotients 
$\widetilde{F}$ of these for which the image of the generator of the center 
is not trivial. By Lemma \ref{lem detect} we may find quantum 
representations for which the order of the image of the center    
can have arbitrary prime power order $p$. Hence, for any prime $p$ 
there are  finite quotients $\widetilde{F}$ of $M_g(1)$  
in which the image of the center has an element of order $p$. 
We apply Lemma \ref{key lemma} to the quotient 
$F$ of $\widetilde{F}$ by the image of the center to get 
finite quotients of the mapping class group with 
arbitrary prime order elements in the second homology.  
\end{proof}

\subsection{Some finite quotients of mapping class groups}
 
\vspace{0.2cm}\noindent Concrete finite quotients with arbitrary torsion 
in their homology can be constructed from mapping class groups 
as follows. Let $p$ be a prime different from $2$ and $3$.  
Recall that we have a linear representation $\widetilde\rho_p:M_g(12)\to U(N(p,g))$ 
which lifts a projective  representation $\rho_p:M_g(12)\to PU(N(p,g))$. 
Moreover, $M_g(1)$ is naturally embedded in $M_g(12)$, by sending the generator of the center 
into 12 times the generator of the center of $M_g(12)$, see \cite{Fu3}.   
According to 
Gilmer and Masbaum \cite{GM} for prime $p$, we have that 
$\widetilde{\rho}_p(M_g(1))\subseteq U(N(p,g))\cap GL(\mathcal O_p)$, where $\mathcal O_p$ is the following ring of cyclotomic integers  
\[
\mathcal O_p=\left\{\begin{array}{ll}
\Z[\zeta_p], & {\rm if} \, p\equiv - 1 ({\rm mod } \, 4)\\
\Z[\zeta_{4p}], & {\rm if} \, p\equiv 1 ({\rm mod } \, 4).\\
\end{array}\right.
\]

\vspace{0.2cm}
\noindent 
Let us then consider the principal ideal $\mathfrak m=(1-\zeta_p)$ 
which is a prime ideal of $\mathcal O_p$.  
It is known that prime ideals of $\mathcal O_p$ are maximal and then 
$\mathcal O_p/{\mathfrak m}^n$ is a finite ring for every $n$. Let then 
$\Gamma_{p, {\mathfrak m}, n}$ be the image of 
$\widetilde{\rho}_{p}(M_g(1))$ into 
the finite group $GL(N(p,g),\mathcal O_p/{\mathfrak m}^n)$ and 
$F_{p, {\mathfrak m}, n}$ be the quotient 
$\Gamma_{p, {\mathfrak m}, n}/\langle \widetilde{\rho}_p(c)\rangle$ by the 
image of the center of $M_g(1)$. We derive a surjective homomorphism 
$\psi_{p,{\mathfrak m}, n}:M_g\to F_{p, {\mathfrak m}, n}$ and a lift 
$\widetilde\psi_{p,{\mathfrak m}, n}:M_g(1)\to \Gamma_{p, {\mathfrak m}, n}$.

\vspace{0.2cm}\noindent
The image $\widetilde{\rho}_p(c)$ of the generator 
$c$ into $\Gamma_{p, {\mathfrak m}, n}$ is the 
scalar root of unity $\zeta_p^{-6}$,  
which is a non-trivial element of  
order $p$ in the ring $\mathcal O_p/{\mathfrak m}^n$ 
and hence an element of order $p$ into 
$GL(N(p,g),\mathcal O_p/{\mathfrak m}^n)$.
Notice that this is a rather exceptional situation, which 
does not occur for other prime ideals in unequal characteristic (see Proposition \ref{unequal}).

\vspace{0.2cm}\noindent
Lemma \ref{key lemma} implies then that the image $(\psi_{p,{\mathfrak m}, n})_*(H_2(M_g))$ within 
$H_2(F_{p, {\mathfrak m}, n})$ contains an element of order $p$. 
This result also shows the contrast between  mapping class group 
representations and  Weil representations: 
\begin{corollary}\label{nontrivial}
If $g\geq 3$, $p$ is prime and $p\not\in\{2,3\}$ (or more generally, 
$p$ does not divide $12$ and not necessarily prime), then 
$\widetilde{\rho}_p({M_g}(1))$ is a non-trivial central extension 
of $\rho_p({M_g})$. Furthermore, under the same hypotheses on $g$ and $p$, 
if $\mathfrak m=(1-\zeta_p)$, then the extension $\Gamma_{p, {\mathfrak m}, n}$ 
of the finite 
quotient $F_{p, {\mathfrak m}, n}$ is non-trivial. 
\end{corollary}
\begin{proof}
The kernel of the homomorphism 
$\widetilde{\rho}_p(\widetilde{M_g}(1))\to \rho_p(M_g)$ is a finite cyclic group of 
$2p$-th roots of unity. Therefore $\ker(\Gamma_{p, {\mathfrak m}, n}\to F_{p, {\mathfrak m}, n})$ is also some finite cyclic group $\nu$.
If the latter extension were trivial then we could find an isomorphism  
$\Gamma_{p, {\mathfrak m}, n}\to F_{p, {\mathfrak m}, n}\times \nu$. Since 
$g\geq 3$ the group $M_g$ is perfect and thus $M_g(1)$ and hence  $\Gamma_{p, {\mathfrak m}, n}$ are also perfect. Thus the projection on the second factor $\Gamma_{p, {\mathfrak m}, n}\to \nu$ 
must be trivial. This contradicts the fact that the image of the generator $c$ of 
$\ker(M_g(1)\to M_g)$ is an element of order $p$ in $\Gamma_{p, {\mathfrak m}, n}$ and hence 
the group of roots of unity $\nu$ has at least order $p$. Thus the 
extension $\Gamma_{p, {\mathfrak m}, n}\to F_{p, {\mathfrak m}, n}$ is nontrivial. 
This argument implies the first claim, as well.  
\end{proof}

\begin{remark}
Although the group $M_2$ is not perfect, because $H_1(M_2)=\Z/10\Z$, 
it still makes sense to consider the central 
extension $\widetilde{M_2}$ arising from the TQFT. 
Then the computations above show that the results of Theorem  \ref{res0} 
and Corollary \ref{nontrivial} hold  for $g=2$ if $p$ is a  prime  and 
$p\not\in\{2,3,5\}$. 
\end{remark}

\begin{remark}
The finite quotients $F_{p, {\mathfrak m}, n}$ associated to 
the ramified principal ideal ${\mathfrak m}=(1-\zeta_p)$ were 
previously considered by Masbaum in \cite{M}. 
\end{remark}

\noindent 
When  $p\equiv -1 ({\rm mod } \, 4)$ the authors of \cite{F4,MaR} 
found many finite  quotients of $M_g$ by using more sophisticated means.
However, the results of \cite{F4,MaR} and the present ones are of a rather different 
nature.  Assume that ${\mathfrak n}$ is a prime ideal of $\mathcal O_p$ such that 
$\mathcal O_p/{\mathfrak n}$ is the finite field
$\mathbb F_q$ with $q$ elements. 
In fact the case of equal characteristics 
$\mathfrak n=\mathfrak m=(1-\zeta_p)$ 
is the only case where non-trivial torsion can arise, according to:  

\begin{proposition}\label{unequal}
If $\mathfrak n$ is a prime ideal of unequal 
characteristic (i.e. such that ${\rm g.c.d.}(p,q)=1$) and $p,q\geq 5$ then  
the image $(\psi_{p,{\mathfrak m}, n})_*(H_2(M_g))$ within 
$H_2(F_{p, {\mathfrak m}, n})$ is trivial. Moreover, for 
all but finitely many prime ideals $\mathfrak n$ of unequal 
characteristic both groups $\Gamma_{p,\mathfrak n, 1}$ and $F_{p,\mathfrak n, 1}$ 
coincide with $SL(N(p,g), \mathbb F_q)$ and hence  $H_2(F_{p,\mathfrak n, 1})=0$.
\end{proposition} 
\begin{proof}
The image of a $p$-th root of unity scalar 
in $SL(N(p,g), \mathbb F_q)$ is trivial, as soon as ${\rm g.c.d.}(p,q)=1$.
Thus  $\Gamma_{p, {\mathfrak n}, n}\to  F_{p,\mathfrak n, n}$ is an isomorphism 
and hence the image of $H_2(M_g)$ into $H_2(F_{p,\mathfrak n, n})$ must be 
trivial. A priori this does not mean that  $H_2(F_{p,\mathfrak n, n})=0$. 
However, 
Masbaum and Reid proved in \cite{MaR} that for all but finitely many  
prime ideals $\mathfrak n$ in  $\mathcal O_p$
the image $\Gamma_{p, {\mathfrak n}, 1}\subseteq  GL(N(p,g),\mathbb F_q)$ 
is the whole group $SL(N(p,g), \mathbb F_q)$.  
It follows that the projection homomorphism 
${M_g}(1)\to SL(N(p,g), \mathbb F_q)$ factors through 
$M_g\to SL(N(p,g), \mathbb F_q)$. But 
$H_2(SL(N, \mathbb F_q))=0$, for $N\geq 4, q\geq 5$, as 
$SL(N, \mathbb F_q)$ itself is the universal central extension 
of $PSL(N, \mathbb F_q)$. 
\end{proof}

\subsection{Property $A_2$ and the proof of Theorem \ref{a2}}
Recall that an equivalent formulation of Serre's property $A_2$ is

\begin{definition}
Let $G$ be a discrete group and $\widehat{G}$ its profinite completion. Then $G$ has  property $A_2$ for the finite $\widehat{G}$-module $M$ if the homomorphism $H^k(\widehat{G}, M)\to H^k(G,M)$ is an isomorphism for
$k\leq 2$ and injective for $k=3$. 
\end{definition}

\begin{proposition}\label{thm centextresfini}
Let $g \geq 4$ be an integer. For any finitely generated  abelian group $A$ and any central extension
\[
 1 \rightarrow A \rightarrow E \rightarrow M_g \rightarrow 1
\]
the group $E$ is residually finite.
\end{proposition}
\noindent 
The key result that  interlocks between Proposition \ref{thm centextresfini} and property $A_2$ is:

\begin{proposition}\label{extresfinite}
\begin{enumerate}
\item A residually finite  group $G$ has property $A_2$ 
for all finite  $\widehat{G}$-modules $M$ if and only if any 
extension by a finite abelian group is residually finite. 
\item Moreover for trivial $\widehat{G}$-modules it is enough to consider central extensions of $G$.
\end{enumerate}
\end{proposition}

\noindent 
Then Theorem \ref{a2} is a consequence of the two propositions above.

\subsubsection{Proof of Proposition \ref{thm centextresfini}}
First we treat the following special case:

\begin{proposition}\label{prop modnisRF}
For any integer  $n \geq 2$, the group $M_g(1)_n$ 
obtained by reducing mod $n$ a generator of the center 
of $M_g(1)$ is residually finite.
\end{proposition}
\begin{proof}
Write $\Z/n\Z$ as a finite product of cyclic groups of prime power order 
$\Z/n\Z =\Z/p_1^{r_1}\Z \times \cdots \times \Z/p_s^{r_s}\Z$. 
Then this isomorphism induces an embedding 
$M_g(1)_n \hookrightarrow M_g(1)_{p_1^{r_1}} \times \cdots \times M_g(1)_{p_s^{r_s}}$, 
and it suffices to prove Proposition 
\ref{prop modnisRF} when $n$ is a prime power. Since  $M_g$ is known to be residually finite~\cite{Gro},
by Malcev's result on the residual finiteness of finitely generated linear groups, 
it is enough to find for each prime power $q^s$ a linear 
representation of the universal 
central extension $M_g(1)$ that factors through 
$M_g(1)_{q^s}$ and detects its center, and this is what Lemma \ref{lem detect} provides.

\end{proof}

\begin{proof}[Proof of Proposition \ref{thm centextresfini}]
We will use below that a group is residually finite if and only if 
finite index subgroups are residually finite. 
 Given our central extension
\[
 1 \rightarrow A \rightarrow E \rightarrow M_g \rightarrow 1
\]
the five term exact sequence in homology reduces to:
\[
 H_2(M_g; \Z) \rightarrow A \rightarrow H_1(E, \Z) \rightarrow 0
\]
because the mapping class group is perfect.
Therefore, any element $f \in E$ that is not in $A$ 
projects non-trivially in the mapping 
class group and is therefore detected by  a finite quotient of 
this group. If $f \in A$ but 
is not in the image of $H_2(M_g;\Z)$, then it projects 
non-trivially into the finitely 
generated abelian group $H_1(E;\Z)$, and is therefore 
detected by a finite abelian 
quotient of $E$. It remains to detect the 
elements in the image of $H_2(M_g;\Z)$. 
Recall the following result:
\begin{lemma}
Let $A$ be a finitely generated abelian group, $B$ a subgroup of $A$.
Then there exists a direct factor  $C$ of $A$ that  
contains $B$ as a subgroup of finite index.
\end{lemma}

\vspace{0.2cm}\noindent 
Apply this lemma to the image $B$ of $H_2(M_{g};\Z)$ into $A$,  
let $p_C$ be the projection onto the subgroup $C$ and consider the 
push-out diagram:

\[
\xymatrix{
1 \ar[r]  & A  \ar[r] \ar^{p_C}[d] & E \ar[r] \ar[d] &  M_g \ar[r]\ar@{=}[d] &  1 \\
1 \ar[r]  & C  \ar[r]  & E_C \ar[r]  &  M_g \ar[r] & 1
}
\]
Then it is sufficient to prove that $E_C$ is residually finite 
in order to show that $E$ is residually finite.

\vspace{0.2cm}\noindent 
Now, the   mapping class group $M_g$ is perfect, and 
therefore we have a push-out diagram:
 
\[
\xymatrix{
1 \ar[r]  & H_2(M_g;\Z)  \ar[r] \ar[d] & M_g(1) \ar[r] \ar[d] &  M_g \ar[r] \ar@{=}[d] &  1 \\
1 \ar[r]  & C  \ar[r]  & E_C \ar[r]  &  M_g \ar[r] & 1,
}
\]
\vspace{0.2cm}\noindent 
where the first row is the universal central  extension, and the 
arrow $ H_2(M_g;\Z) \rightarrow C$ is the one appearing in 
the five term exact sequence of the bottom extension.  Recall that for $g \geq 4$, 
$H_2(M_g; \Z) = \Z$.  Two cases could occur:
 \begin{enumerate}
  \item Either $ H_2(M_g;\Z) \rightarrow C$ is injective and in this case $E_C$ 
contains the residually finite group $M_g(1)$ as a subgroup of finite index, and 
this is known to be residually finite (see \cite{Fu3}).
  \item Or the image of $ H_2(M_g;\Z) \rightarrow C$ is a cyclic group $\Z/k \Z$ and 
$E_C$ contains as a finite index subgroup the reduction mod $k$ of the universal central 
extension, and we  conclude by applying Proposition \ref{prop modnisRF}.
 \end{enumerate}
\end{proof}

\subsubsection{Proof of Proposition \ref{extresfinite} (1)}
Assume that every  extension   of $G$ by a finite  abelian group is residually finite. 
Let $x \in H^2(G;A)$ be represented by the extension:
\[
\xymatrix{
 1 \ar[r]  & A  \ar[r]  & E \ar[r]  &  G \ar[r] & 1.
}
\]
By the equivalent property $D_2$ (see the next subsection), 
it is enough to find a finite index subgroup $H \subseteq G$ 
such that $x$ is zero in $H^2(H;A)$. 
Observe that property $A_1$, and thus $D_1$, is automatic. 
Since $E$ is residually finite, 
for each  non-trivial 
element $a \in A$ choose a finite quotient $F_a$ of $E$ in which 
the image of $a$ is not identity. 
Let $B_a$ be 
the image of $A$ in $F_a$, and $Q_a  = F_a/B_a$. Denote by $F_A, B_A$ and 
$Q_A$ the products of 
these finitely many groups over the non-trivial elements in $A$.   Then 
the diagonal map  
$E \rightarrow F_A$  fits into a commutative diagram:
\[
\xymatrix{
 1 \ar[r]  & A  \ar[r] \ar[d] & E \ar[r] \ar[d]  &  G \ar[r] \ar[d] & 1 \\ 
 1 \ar[r]  & B_A  \ar[r]  & F_A \ar[r]  &  Q_A \ar[r] & 1
 }
\]
Let $K$ be the kernel $\ker(G \rightarrow Q_A)$.
Then $K$ is a finite index normal subgroup 
and the pull back of $x$ to $H^2(K; A)$ is trivial.

\vspace{0.2cm}\noindent 
Conversely, assume that the residually finite group  $G$ has property $A_2$ and let 
\[
\xymatrix{
 1 \ar[r]  & A  \ar[r]  & E \ar[r]  &  G \ar[r] & 1
}
\]
be an extension of $G$ by the finite abelian group $A$.  Then, by \cite[Ex. 2 Ch. I.2.6]{Serre}, we have a natural short exact sequence of profinite completions:
\[
\xymatrix{
 1 \ar[r]  & \widehat{A}  \ar[r]  & \widehat{E} \ar[r]  &  \widehat{G} \ar[r] & 1
}
 \]
that fits into a commutative diagram 
\[
\xymatrix{
 1 \ar[r]  & A  \ar[r] \ar[d] & E \ar[r] \ar[d]  &  G \ar[r] \ar[d] & 1 \\
 1 \ar[r]  & \widehat{A}  \ar[r]  & \widehat{E} \ar[r]  &  \widehat{G} \ar[r] & 1
}
 \]
Since $A$ is finite $A \simeq \widehat{A} $, and since $G$ is residually finite 
the rightmost downward arrow is an injection. 
By the five lemma the homomorphism  
$E \hookrightarrow \widehat{E}$ is also injective, 
and hence $E$ is  residually finite as it is a subgroup of a 
profinite group.

\subsubsection{Proof of proposition \ref{extresfinite} (2) and property $E_n$}\label{propertyE}
It would be nice, but probably difficult, to understand under which 
assumptions  property $A_2$ for a residually finite, 
finitely presented group $G$ and all finite trivial $\widehat{G}$-modules 
implies property $A_2$. However, there exists a stronger related 
condition on groups for which this kind of statement will hold. 
Serre introduced several  properties in \cite[Ex. 1, I.2.6]{Serre} 
as follows. One says that a residually finite group $G$ has property: 

\begin{enumerate}
\item $(A_n)$ if $H^j(\widehat{G}; M)\to H^j(G;M)$ is bijective for 
$j\leq n$ and injective for $j=n+1$, for all finite $\widehat{G}$-modules $M$. 

\item $(B_n)$ if $H^j(\widehat{G}; M)\to H^j(G;M)$ is surjective for 
$j\leq n$ and  for all finite $\widehat{G}$-modules $M$. 

\item $(C_n)$ if for each  finite $\widehat{G}$-module $M$ and $x\in H^j(G; M)$, $1\leq j\leq n$, there exists 
a discrete $\widehat{G}$-module $M'$  containing $M$ such that 
the image of $x$ in $H^j(G;M')$ is zero.

\item $(D_n)$ if for each finite $\widehat{G}$-module $M$ and $x\in H^j(G; M)$, $1\leq j\leq n$, there exists a 
subgroup $H\subseteq G$ of finite index in $G$ such that 
the image of $x$ in $H^j(H;M)$ is zero.

 \item $(E_n)$ if $\widehat{H}^j(G;M)=0$, for $1\leq j\leq n$ and 
for all finite $\widehat{G}$-modules $M$.
\end{enumerate}
Then Serre stated that properties $A_n$, $B_n$, $C_n$ and $D_n$ are 
equivalent. It is easy to see that these properties are also equivalent when 
we fix the $\widehat{G}$-module $M$, or we let it run over the 
finite trivial  $\widehat{G}$-modules.  

\vspace{0.2cm}\noindent 
Denote by $\widehat{H}^n(G;M)={\lim}_{H \subset_f G} H^n(H;M)$, 
where the {\em direct} limit is taken with respect with the 
directed set of $H\subset_fG$, meaning that 
$H$ is a finite index subgroup of $G$. The  directed set of 
inclusions homomorphisms induces a homomorphism 
\[ H^n(G;M)\to \widehat{H}^n(G;M)\]
Note that if the homomorphisms 
$H^j(G;M)\to \widehat{H}^j(G;M)$ have zero image, for $1\leq j\leq n$, then 
 condition $(D_n)$ is satisfied.  Conversely, if condition $(D_n)$ is verified, 
 and $H^j(G;M)$ are finite for $1\leq j\leq n$, then 
 there exists some finite index subgroup $H\subset_fG$ such that 
 the image of the restriction  homomorphism $H^j(G;M)\to H^j(H;M)$ is trivial. 
By the universality of the direct limit the homomorphism 
$H^j(G;M)\to \widehat{H}^j(G;M)$ factors through $H^j(H;M)$ and hence its image is zero. 
In any case property $(D_n)$ is 
a consequence of property $(E_n)$. 
An interesting fact 
concerning the latter is the following:

\begin{proposition}\label{nontrivialE}
If a residually finite group $G$ has property $E_n$ for 
all finite trivial $\widehat{G}$-modules then it has property $E_n$. 
\end{proposition}
\begin{proof}
First we can easily step from $\mathbb F_p$-coefficients to any trivial $G$-module. 
\begin{lemma}\label{fp}
Condition $(D_n)$ for $G$ and all finite trivial $\widehat{G}$-modules $\mathbb F_p$ 
implies  $(D_n)$ for $G$ and all finite trivial $\widehat{G}$-modules $M$. 
\end{lemma}

\vspace{0.2cm}\noindent 
The second ingredient allows us to pass from all trivial coefficients 
to arbitrary coefficients:
\begin{lemma}\label{arbit}
Condition $(E_n)$ for $G$ and all finite trivial $\widehat{G}$-modules $M$ 
implies  $(E_n)$ for $G$ and all finite trivial $\widehat{G}$-modules $M$. 
\end{lemma}

\vspace{0.2cm}\noindent 
This proves the claim of Proposition \ref{nontrivialE}.
\end{proof}

\vspace{0.2cm}\noindent 
In the proof of Lemma \ref{arbit} we will make use of the following 
rather well-known result:
\begin{lemma}\label{limits}
If $J\subseteq I$ are two directed sets such that 
$J$ is cofinal in $I$, then for any direct system of abelian groups 
$(A_i, f_{ij})$ indexed by $I$ we have 
\[ \lim_{\alpha,\beta\in J} (A_{\alpha}, f_{\alpha\beta})=0, \;\;\; 
{\rm if \; and \;  only \;  if }  \;\;
\lim_{i,j\in I} (A_{i}, f_{ij})=0.\]
\end{lemma}
\begin{proof}[Proof of Lemma \ref{fp}]
This follows from decomposing the finite trivial $\widehat{G}$-module, i.e. a finite abelian group 
$A$, into $p$-primary components and then use induction on the rank of 
the composition series of $A$ and the 5-lemma.
\end{proof} 
\begin{proof}[Proof of Lemma \ref{arbit}]
Let now $M$ be an arbitrary finite $G$-module. 
Let $K_M$ be the kernel of the $G$-action on $M$. 
We denote by $A$ the trivial $G$-module which is isomorphic as an abelian 
group to $M$. 
By hypothesis condition $(E_n)$ is satisfied for the group $G$ 
and the trivial module $A$, so that $\widehat{H}^j(G; A)=0$, $1\leq j\leq n$.
Since $M$ is finite $K_M$ is of finite index in $G$.
Consider the set  $J_M=\{H\subset_f K_M\}$ of finite index subgroups of 
$K_M\subseteq G$. This is a subset of the  directed set $I_G$
of finite index subgroups of $G$. Further $J_M$ is cofinal in $I_G$ 
with respect to the inclusion as any subgroup $H\in I_G$ 
contains the subgroup $H\cap K_M\subset_f K_M$. 
 According to Lemma \ref{limits} we have then  
$\lim_{H\subset_f K_M} H^j(H; A)=0$.
Furthermore, for each $H\subseteq K_M$ the  $H$-modules $A$ and $M$ are 
isomorphic as both are trivial. Thus there exists a canonical family of 
isomorphisms $i_H:H^j(H; A)\cong H^j(H;  M)$ which is compatible with 
the direct structures on  the cohomology groups 
indexed by $J_M=\{H\subset_f K_M\}$.
We have therefore $\lim_{H\subset_f K_M} H^j(H;  M)=0$. 
However using again Lemma \ref{limits} for the sets $J_M$ and $I_G$  
in the reverse direction we obtain $\widehat{H}^j(G; M)=0$, $1\leq j\leq n$.
 \end{proof}

\vspace{0.2cm}\noindent 
\begin{proof}[End of proof of Proposition \ref{extresfinite} (2)]
Let $G$ be a residually finite group. Every $x\in H^2(G;\mathbb F_p)$ 
is represented by a central extension $E$ of $G$ by $\Z/p\Z$. 
By the proof of Proposition  \ref{extresfinite} (1) $E$ is residually finite if and 
only if  there exists a finite index subgroup $H\subseteq G$ such that 
the image of $x$ in $H^2(H;\mathbb F_p)$ is zero. 
By Lemma \ref{fp} this is equivalent to the group $G$ having property $D_2$ for all 
trivial $\widehat{G}$-modules $M$. 
Therefore, $G$ has  property $A_2$  for all trivial 
$\widehat{G}$-modules $M$ if all central extensions by $\Z/p\Z$ are residually finite, 
for all primes $p$. 
\end{proof}

\begin{remark}
The analog of Lemma~\ref{fp} holds also for property $A_n$, with the same proof. 
Hovewer, this is not clear for Lemma \ref{arbit}.
\end{remark}

\begin{remark}
The discussion about property $E_n$ clarifies some of the statements 
in \cite{GJZ}. Specifically, Lemmas 3.1. and 3.2. 
from \cite{GJZ} concern only property $E_n$ instead of property $A_n$. 
Nevertheless, the main result of  \cite{GJZ} is valid with the same proof. 
\end{remark}

\subsection{Proof of Corollary \ref{vtrivialmcg}}

\begin{lemma}\label{vtrivial}
An extension $E$ of the residually finite group $G$ by a finite group $A$ is residually finite if and only if it is virtually trivial. 
\end{lemma}
\begin{proof}
If $E$ is residually finite then there is a finite index 
subgroup $\Delta\subset E$ with $\Delta\cap A=\{1\}$. Then the projection $E\to G$ 
restricts to an isomorphism on $\Delta$ and hence the extension splits over the image of $\Delta$ in $G$.  Conversely, if the extension $E$ splits over a finite index subgroup $\Delta\subseteq G$ 
then $\Delta\subset E$ is residually finite and of finite index in $E$ and hence 
$E$ must be residually finite, as well.   
\end{proof}

\begin{lemma}\label{finitepullback}
An extension $E$ of the group $G$ by a finite group $A$ is the pull-back 
of some extension of a finite group $F$ by some surjective homomorphism 
$f:G\to F$ if and only if it is virtually trivial. 
\end{lemma}
\begin{proof}
If $E$ splits over the finite index subgroup $\Delta\subseteq G$, then the image 
$\Delta\subseteq E$ of a section intersects $A$ trivially. By passing to a finite index 
subgroup of $\Delta$  we can assume that $\Delta$ is a normal subgroup of $E$. 
Then the extension $E$ is the pull-back of the extension 
\[ 1\to A \to E/\Delta\to G/\Delta\to 1\] 
In the reverse direction, a 
pull-back of an extension by $f:G\to F$ is split over the finite index subgroup $\ker f$.  
\end{proof}

\vspace{0.2cm}\noindent 
Now, Proposition \ref{thm centextresfini} along with Lemmas \ref{vtrivial} and \ref{finitepullback} imply   Corollary \ref{vtrivialmcg}.

\section{Towards an inductive proof of Theorem \ref{tors-sympl0}}\label{computed}
\subsection{Motivation}\label{motiv} 
For a prime $p$, an integer $k \geq 1$ we have two fundamental extensions:
\[
1 \rightarrow Sp(2g,\Z,p^k) \rightarrow Sp(2g,\Z) \rightarrow Sp(2g,\Z/p^k\Z) \rightarrow 1,  
\]
and
\begin{equation}\label{fdt2}
1 \rightarrow \mathfrak{sp}_{2g}(p) \rightarrow Sp(2g, \Z/p^{k+1} \Z) \rightarrow Sp(2g, \Z/p^k\Z) \rightarrow 1. 
\end{equation}
In particular, every element in 
$Sp(2g,\Z,p^k)$ can be written as $\mathbf 1_{2g} + p ^kA$, for some matrix 
$A$ with integer entries. If the symplectic form is written 
as $J_g  = \left( \begin{matrix} 
0 & \mathbf 1_g \\ 
-\mathbf 1_g & 0      
\end{matrix}
\right)$ 
then the matrix $A$ satisfies the equation 
$A^{\top} J_g + J_g A \equiv 0 ({\rm mod }\; p)$. 
Then  we set $\mathfrak{sp}_{2g}(p)$ for  the additive group of those  
matrices with entries in $\Z/p\Z$ that satisfy the 
equation $M^{\top} J_g + J_g M \equiv 0 ({\rm mod } \; p)$. In particular this subgroup  is  independent of the integer $k$.

\vspace{0.2cm}
\noindent 
The  homomorphism $j_q: Sp(2g,\Z, p ^k) \to \mathfrak{sp}_{2g}(p)$ sending  
$\mathbf 1_{2g} + p^k A$ onto  $A \; ({\rm  mod }\; p)$ 
is surjective (see \cite{Sato}).

\vspace{0.2cm}
\noindent 
The different actions of the symplectic group $Sp(2g,\Z)$ that 
$\mathfrak{sp}_{2g}(p)$ inherits from these descriptions  coincide.  
We will use in this text the action that is induced by 
the conjugation action on $Sp(2g,\Z,p)$ via the surjective map $j_q$. 
Notice that clearly this action factors through $Sp(2g,\Z/p\Z)$.

\vspace{0.2cm}\noindent 
The second page of the Hochschild-Serre spectral sequence associated to the exact sequence (\ref{fdt2}) 
in low degrees is as follows:
\begin{equation}
\begin{array}{|c|c|c|}
 H_2(\mathfrak{sp}_{2g}(p))_{Sp(2g,\Z/p^k\Z)} & & \\
0 & H_1(Sp(2g,\Z/p^k\Z); \mathfrak{sp}_{2g}(p)) &  \\
 \Z& 0 &H_2(Sp(2g,\Z/p^k\Z))  \\
 \hline
\end{array}
\end{equation}

\vspace{0.2cm}\noindent 
In fact from \cite[Lemma 3.5]{FP} we know that
 \[
H_1(\mathfrak{sp}_{2g}(p))_{Sp(2g,\Z/p^k\Z)} = 
H_1(\mathfrak{sp}_{2g}(p))_{Sp(2g,\Z/p\Z)}=0.
\]
as the action of $Sp(2g,\Z/p^k\Z)$ on $\mathfrak{sp}_{2g}(p)$ factors 
through the action of  $Sp(2g,\Z/p\Z)$. 

\vspace{0.2cm}\noindent 
Thus the calculations needed for an inductive 
computation of $H_2(Sp(2g,\Z/2^k\Z))$  are  
the result in Theorem \ref{lem H2sp(p)} and the following 
theorem of Putman, see \cite[Thm. G]{Pu1}:

\begin{theorem}\label{lem annulhomol}
For any odd prime $p$, any integer $k \geq 1$ and any $g \geq 3$ we have:
\begin{equation}
H_1(Sp(2g,\Z/p^{k}\Z), \mathfrak{sp}_{2g}(p))= 0.
\end{equation}
Moreover, this holds true also when $p=2$ and $k=1$. 
\end{theorem}

\noindent 
Unfortunately we do not know whether 
$H_1(Sp(2g,\Z/2^k\Z); \mathfrak{sp}_{2g}(2))=0$ or not for $k\geq 2$. 
However this is true when $k=1$ and we derive:  

\begin{corollary}\label{inductive}
Assume Theorem \ref{lem H2sp(p)} holds. 
Then $H_2(Sp(2g,\Z/4\Z))\in \{0,\Z/2\Z\}$, for all $g\geq 4$. 
\end{corollary}  
\begin{proof}
Proposition \ref{lem H2epi} implies that $H_2(Sp(2g,\Z/4\Z))$ is cyclic, 
for $g\geq 4$. Since $H_2(Sp(2g,\Z/2\Z))=0$, and 
$H_1(Sp(2g,\Z/2\Z), \mathfrak{sp}_{2g}(2))= 0$ from  
Putman's theorem \ref{lem annulhomol}, the only non-zero  
term of the second page of the Hochschild-Serre spectral sequence above 
computing the cohomology of $Sp(2g,\Z/4\Z)$ is 
$H_2(\mathfrak{sp}_{2g}(2))_{Sp(2g,\Z/2\Z)}$. 
Then, by Theorem \ref{lem H2sp(p)} the rank of  $H_2(Sp(2g,\Z/4\Z))$ 
is at most $1$, which proves the claim.  
\end{proof}

\subsection{Generators for the module $\mathfrak{sp}_{2g}(p)$}
We describe first a small set of generators of $\mathfrak{sp}_{2g}(p)$ 
as an $Sp(2g,\Z)$-module. Denote by $\mathfrak M_g(R)$ the $R$-module of  
$g$-by-$g$ matrices with entries from the ring $R$. 
A direct computation shows that a matrix 
$
\left( 
\begin{matrix} 
A & B \\
C & D
\end{matrix}
\right) 
\in \mathfrak M_{2g}(\Z/p\Z)$ written by blocks is in $\mathfrak{sp}_{2g}(p)$ if 
and only if $A + D^{\top}\equiv 0 \; ({\rm mod}\; p)$ and both $B$ and 
$C$ are symmetric matrices. 
It will be important for our future computations to keep in mind 
that the subgroup $GL(g, \Z) \subset Sp(2g,\Z)$ 
preserves this decomposition into blocks. 
From this description we immediately get a set of generators 
of $\mathfrak{sp}_{2g}(p)$ as an additive group. 
Recall that $e_{ij}\in \mathfrak M_g(R)$  denotes the elementary  
matrix whose only non-zero coefficient is $1$ at the place $ij$.  
Define now the following matrices in $\mathfrak{sp}_{2g}(p)$ for $i, j\in\{1,2,\ldots,g\}$:
\begin{equation}
 u_{ij} = \left( \begin{matrix}
                  0 & e_{ij}+e_{ji} \\
                  0 & 0
                 \end{matrix}
 \right), \quad u_{ii}  = \left( \begin{matrix}
                 0 & e_{ii} \\
                  0 & 0
                 \end{matrix}
 \right), \quad l_{ij} = \left( \begin{matrix}
                 0 & 0 \\
                  e_{ij}+e_{ji}  & 0
                 \end{matrix}
 \right), \quad l_{ii}
= \left( \begin{matrix}
                  0 & 0 \\
                  e_{ii}  & 0
                 \end{matrix}
 \right),
\end{equation}
\begin{equation}
r_{ij} = \left( \begin{matrix}
                 e_{ij} & 0 \\
                  0  & -e_{ji}
                 \end{matrix}
 \right), \quad n_{ii} 
= \left( \begin{matrix}
                  e_{ii} & 0 \\
                  0  & -e_{ii}
                 \end{matrix}
 \right).
\end{equation}
 
\vspace{0.2cm}\noindent 
Therefore we have: 

\begin{proposition}\label{prop genspmod}
As an $Sp(2g,\Z/p\Z)$-module, $\mathfrak{sp}_{2g}(p)$ is generated 
by $r_{ij}, n_{ii}, u_{ij}$, $u_{ii}$, $l_{ij}$ and $l_{ii}$, where $i, j\in \{1,2,\ldots,g\}$.
\end{proposition}

\vspace{0.2cm}\noindent 
And as $GL({g}, \Z/p\Z)$-module we have:
\begin{lemma}\label{lem glgmodule}
Let ${\rm Sym}_g(\Z/p\Z)\subset\mathfrak M_g(\Z/p\Z)$ denote the submodule 
of symmetric matrices.  
We have an identification of  $GL(g, \Z)$-modules:
\[
 \mathfrak{sp}_{2g}(p) = \mathfrak M_{g}(\Z/p \Z) \oplus {\rm Sym}_g(\Z/p\Z ) \oplus 
{\rm Sym}_g(\Z/p\Z). 
\]
The action of  $GL(g, \Z/p\Z)$ on $\mathfrak M_g(\Z/p\Z)$ is by conjugation, 
the action on the first copy ${\rm Sym}_g(\Z/ p\Z)$ 
is given by $x\cdot A = x A x^{\top}$ and on the second copy ${\rm Sym}_g(\Z/ p\Z)$
is given by $x\cdot A = (x^{\top})^{-1} A x^{-1}$. 

\vspace{0.2cm}
\noindent 
A set of generators for $\mathfrak M_{g}(\Z/p \Z)$ is given by the set of elements $r_{ij}$ and $n_{ii}$, $1\leq i, j \leq g$, $i \neq j$. 
The two copies of ${\rm Sym}_g(\Z/p\Z )$ are generated by 
the matrices $l_{ij}$ and $u_{ij}$ respectively, where $1\leq i, j\leq g$.
\end{lemma}

\subsection{Proof of Theorem \ref{lem H2sp(p)} }

\vspace{0.2cm}\noindent 
Notice first that as the action of $Sp(2g,\Z/p^k\Z)$ 
factors through $Sp(2g,\Z/p\Z)$ via the mod $p^{k-1}$ reduction map, 
for any $k \geq 2$ we have 
$H_{2}(\mathfrak{sp}_{2g}(p);\Z)_{Sp(2g,\Z/p^k\Z)} \simeq H_{2}(\mathfrak{sp}_{2g}(p))_{Sp(2g,\Z/p\Z)}$.  Also, as $\mathfrak{sp}_{2g}(p)$ 
is an abelian group there is a canonical isomorphism:

\begin{equation}
H_{2}(\mathfrak{sp}_{2g}(p))= \wedge^2 \mathfrak{sp}_{2g}(p).
\end{equation}
Let $M$  denote the set of elements $r_{ij}$ and $n_{ii}$, $1\leq i, j \leq g$, $i \neq j$ and 
$S$ denote the set of elements $l_{ij}$ and $u_{ij}$, where $1\leq i, j\leq g$.
The group $\wedge^2 \mathfrak{sp}_{2g}(p)$ is generated by  the set 
of exterior powers of pairs of generators given in Proposition~\ref{prop genspmod}, which we split naturally into three  
disjoint subsets:
\begin{enumerate}
\item The subset $S \wedge S$ of exterior powers $u_{ij} \wedge l_{kl}$, $l_{ij} \wedge l_{kl}$, $u_{ij} \wedge u_{kl}$.
\item The subset $S \wedge M$ of exterior power $l_{ij} \wedge n_{kk}$, $u_{ij} \wedge n_{kk}$, $l_{ij}\wedge r_{kl}$ and $u_{ij}\wedge r_{kl}$.
\item The subset $M \wedge M$ of exterior powers  
$r_{ij} \wedge r_{kl}$, $r_{ij} \wedge n_{kk}$ and $n_{jj}\wedge n_{ii}$.  
\end{enumerate}

\vspace{0.2cm}\noindent
We will first show that the image of $S \wedge S$ and $S \wedge M$ is $0$ in 
$\wedge^2 \mathfrak{sp}_{2g}(p)$, and in a second time we will show how 
the $\Z/2\Z$ factor appears in the image of $M \wedge M$. It will also be clear from the proof why there is no such non-trivial element in odd characteristic.

\vspace{0.2cm}\noindent
 We use constantly the trivial fact  that the action of $Sp(2g,\Z/p\Z)$ is trivial on the 
coinvariants module. So to show the nullity of the image of a generator it is enough to show that in each orbit of 
a generating element of $S \wedge S$ or $S \wedge M$ there is the $0$ element. We will in particular heavily 
use the fact that the symmetric group 
$\mathfrak{S}_g \subset GL(g,\Z/p\Z) \subset Sp(2g,\Z/p\Z)$ 
acts on the basis elements by permuting the indices. 
Notice that the symmetric group action 
respects the above partition into three sets of elements. 
To emphasize when we use such a permutation to identify 
two elements in the co-invariant module we will use the notation 
$\doteq$ instead of $=$.

\paragraph{Nullity of the generators in $S \wedge S$.}
Picking one representative in each $\mathfrak{S}_g$-orbit we are left 
with the following elements. Here ``Type'' refers to the number of 
distinct indexes that appear in the wedge, as this is the only thing 
that really matters. Of course type $IV$ elements appear only for 
$g \geq 4$. 

\begin{enumerate}
	\item Type I:  $u_{11} \wedge l_{11}$.
	\item Type II: $u_{11} \wedge u_{22}$, $l_{11} \wedge l_{22}$, $u_{11} \wedge l_{22}$, $u_{11} \wedge u_{12}$, $l_{11} \wedge l_{12}$, $u_{12} \wedge l_{12}$, $u_{12} \wedge l_{22}$, $u_{11} \wedge l_{12}$.
	\item Type III: $u_{11} \wedge u_{23}$, $l_{11} \wedge l_{23}$, $u_{12} \wedge u_{23}$, $l_{12} \wedge l_{23}$, $u_{11} \wedge l_{23}$, $l_{11} \wedge u_{23}$, $u_{12} \wedge l_{23}$.
	\item Type IV: $u_{12} \wedge u_{34}$, $l_{12} \wedge l_{34}$, $u_{12} \wedge l_{34}$.
\end{enumerate}

\vspace{0.2cm}\noindent 
Using the fact that $J_g \cdot u_{ij} = -l_{ij}$, one can identify 
some generators in $S\wedge S$, for instance $u_{11} \wedge u_{22} = l_{11} \wedge l_{22}$, and we are left with:

\begin{enumerate}
	\item Type I:  $u_{11} \wedge l_{11}$.
	\item Type II: $u_{11} \wedge u_{22}$ , $u_{11} \wedge l_{22}$, $u_{11} \wedge u_{12}$, $u_{11} \wedge l_{12}$, $u_{12} \wedge l_{12}$, $u_{12} \wedge l_{22}$.
	\item Type III: $u_{11} \wedge u_{23}$, , $u_{12} \wedge u_{23}$, $u_{11} \wedge l_{23}$, $u_{12} \wedge l_{23}$.
	\item Type IV: $u_{12} \wedge u_{34}$, $u_{12} \wedge l_{34}$.
\end{enumerate}

\vspace{0.2cm}\noindent 
We consider now two families of elements in the symplectic group, for $1\leq i\neq j\leq g$:
\[
\tau^u_{ij} = 
\left( 
\begin{matrix}
1_g & e_{ii} + e_{jj} \\
0 & 1_{g}
\end{matrix}
\right), \quad
\tau^l_{ij} = 
\left( 
\begin{matrix}
1_g & 0 \\
e_{ii} + e_{jj}  & 1_{g}
\end{matrix} 
\right).
\]
A direct computation shows that:
\begin{equation}
\tau^u_{ij}\cdot u_{k\ell} = u_{k\ell},  \quad \tau^u_{ij} \cdot r_{ij} = r_{ij} - u_{ij}, \: {\rm for \: all } \:  i, j, k, \ell
\end{equation}
\begin{equation}
\tau^l_{ij}\cdot l_{k\ell} = l_{k\ell},  \quad \tau^l_{ij}\cdot r_{ij} = r_{ij} + l_{ij}, \: {\rm for \: all } \: i, j, k, \ell
\end{equation}
In particular, we obtain: 
\begin{equation}
\tau^u_{ij}\cdot (u_{k\ell} \wedge r_{ij}) = u_{k\ell}\wedge r_{ij} - u_{k\ell} \wedge u_{ij}
\end{equation}
and
\begin{equation}
\tau^l_{ij}\cdot (l_{k\ell} \wedge r_{ij}) = l_{k\ell}\wedge r_{ij} + l_{k\ell} \wedge l_{ij}
\end{equation}
This shows that all elements of the form $u \wedge u$ or 
$l \wedge l$ vanish in the module of coinvariants, except possibly for 
$u_{11} \wedge u_{22}$  and $l_{11} \wedge l_{22}$. We are left then  with:

\begin{enumerate}
	\item Type I:  $u_{11} \wedge l_{11}$.
	\item Type II: $u_{11} \wedge u_{22}$ , $u_{11} \wedge l_{22}$, $u_{11} \wedge l_{12}$, $u_{12} \wedge l_{12}$, $u_{12} \wedge l_{22}$.
	\item Type III: $u_{11} \wedge l_{23}$, $u_{12} \wedge l_{23}$.
	\item Type IV: $u_{12} \wedge l_{34}$.
\end{enumerate}

\vspace{0.2cm}\noindent 
Consider the exchange map $E_{ij}$, $i \neq j$, defined by $a_i \rightarrow -b_i$, $b_i \rightarrow a_i$, $a_j \rightarrow -b_j$, $b_j \rightarrow a_j$, and all other basis elements fixed. By construction, these maps act as follows: 
$E_{ij}\cdot u_{ij} = - l_{ij}$, $E_{ij}\cdot u_{ii} = -l_{ii}$, 
$E_{ij} \cdot u_{jj} = -l_{jj}$ and $E_{ij} \cdot u_{kl} = u_{kl}$ 
if $\{i,j\} \cap \{k,l \}=\emptyset$.   

\vspace{0.2cm}\noindent 
Denoting in the same way the (trivial!) action on the  quotient module of  coinvariants $(\wedge^2 \mathfrak{sp}_{2g}(p))_{Sp_2(2g,\Z/p\Z)}$, we find:
\begin{enumerate}
	\item $u_{11} \wedge l_{23} = E_{23}\cdot (u_{11} \wedge l_{23}) =- u_{11} \wedge u_{23} =0$. 
	\item $u_{12} \wedge l_{22} = E_{12}\cdot (u_{12} \wedge l_{22}) 
	= -l_{12} \wedge -u_{22} = u_{11} \wedge l_{12}$.
	\item $u_{11} \wedge u_{22} = E_{23}\cdot (u_{11} \wedge u_{22}) = - u_{11} \wedge l_{22}$.
	\item $u_{12} \wedge l_{34} = E_{34} \cdot (u_{12} \wedge l_{34}) = u_{12} \wedge -u_{34} = 0$.
\end{enumerate}

\vspace{0.2cm}\noindent 
Thus the coinvariants module is generated by the classes of the following elements:

\begin{enumerate}
	\item Type I: $u_{11} \wedge l_{11}$.
	\item Type II: $u_{11} \wedge l_{22}$, $u_{11} \wedge l_{12}$, $u_{12} \wedge l_{12}$.
	\item Type III: $u_{12} \wedge l_{23}$.
\end{enumerate}

\vspace{0.2cm}\noindent 
We now use the action of the following symplectic maps, for $1\leq i\neq j\leq g$:
\[
A_{ij} = \left(
\begin{matrix}
1_g -e_{ji} & 0  \\
0 & 1_g + e_{ij}
\end{matrix}
\right). 
\]
\vspace{0.2cm}\noindent 
By direct computation we obtain: 
\[A_{ij} \cdot u_{ii} = u_{ii} + u_{jj} - u_{ij}, \, A_{ij} \cdot u_{jj} = u_{jj}, \,A_{ik} \cdot u_{ij} = u_{ij} - u_{jk},\, A_{ki} \cdot u_{ij} = u_{ij}\]
and
\[
A_{ij} \cdot l_{ii} = l_{ii} + l_{jj} - l_{ij}, \, A_{ij} \cdot l_{jj} = l_{jj}, \, A_{ik} \cdot l_{ij} = l_{ij} - l_{jk}, \, A_{ki} \cdot l_{ij} = l_{ij}
\]
for pairwise distinct values of the  indices $i,j,k$. Further, we derive the following 
equalities that hold in the  quotient module of coinvariants:
\begin{enumerate}
	\item $u_{12} \wedge l_{23} = A_{13} \cdot (u_{12} \wedge l_{23}) = (u_{12} - u_{23}) \wedge l_{23}$, so $0= u_{23} \wedge l_{23} \doteq u_{12} \wedge l_{12}$.
	\item $ u_{22} \wedge l_{32}= A_{13} \cdot (u_{22} \wedge l_{32}) = u_{22} \wedge (l_{32} + l _{12})$, so $u_{11} \wedge l_{12} \doteq u_{22} \wedge l_{12} = 0$. 
	\item $u_{12} \wedge l_{34} = A_{23} \cdot (u_{12} \wedge l_{34}) = u_{12}  \wedge(l_{34} + l_{24})$ , so $u_{12} \wedge l_{23} \doteq u_{12} \wedge l_{24} = 0$.
	\item $u_{11} \wedge l_{11} = A_{12} \cdot (u_{11} \wedge l_{11}) = (u_{11} + u_{22} - u_{12}) \wedge l_{11}$, so \\ $u_{22} \wedge l_{11} = u_{12} \wedge l_{11}  = E_{12}\cdot (l_{12} \wedge u_{11}) = 0$.
	\item $u_{22} \wedge l_{11} = A_{12} \cdot (u_{22} \wedge l_{11}) = u_{22} \wedge (l_{11} + l_{12} + l_{22})$ so \\
	$u_{11} \wedge l_{11} \doteq u_{22} \wedge l_{22} = -u_{22} \wedge l_{12} \doteq  - u_{11} \wedge l_{12} = 0$.  
\end{enumerate}
This finishes the computation.

\paragraph{Nullity of the generators in $S \wedge M$.}
Here we separate between two types of generators:
\begin{enumerate}
\item Generators of the form $u_{ks} \wedge n_{ii}$ and $l_{ks} \wedge n_{ii}$ for arbitray $k,s,i$.

We let the matrix $\tau_{ij}^\ell$ of the previous section act on these generators. A direct computation shows that, for arbitrary values of $k,s,i,j$ with $i \neq j$, we have:
\[
\tau_{ij}^\ell \cdot l_{ks} = l_{ks} \text{ and } \tau_{ij}^\ell \cdot u_{jj} = u_{jj} - l_{jj} - n_{jj}.
\]

Therefore, relying on our previous computations:
\[
0 = \tau_{ij}^\ell \cdot (\ell_{ks} \wedge u_{jj}) = l_{ks} \wedge (u_{jj}- l_{jj} - n_{j}) = - l_{ks} \wedge n_{jj}
\]

To get the nullity for elements of the form $u \wedge n$, we apply $J_g$ to the previous element, and use the fact that, up to a sign, $J_g$ exchanges the elements $u_{ij}$ and  $l_{ij}$  while fixing $n_{ii}$.
\item For the elements of the form  $u_{ks} \wedge r_{ij}$ and $l_{ks} \wedge r_{ij}$ for arbitrary $k,s,i,j$ and $i \neq j$, we apply the element
\[
\tau_i^\ell = \left( \begin{matrix}
1_g & 0 \\e_{ii} & 1_g
\end{matrix}\right).
\]
By direct computation we obtain:
\[
\tau_i^ \ell \cdot u_{ij} = u_{ij} - r_{ij} + l_{ii} \text{ and } \tau_i^\ell \cdot l_{ks} = l_{ks}.
\]
Therefore, for $i \neq j$ and $k,s$ arbitrary we find:
\[
0 = \tau_i^\ell \cdot (u_{ij} \wedge l_{ks}) = (u_{ij} - r_{ij} + l_{ii}) \wedge l_{ks} = -r_{ij} \wedge l_{ks}.
\]
By applying $J_g$ and using that $J_g \cdot r_{ij} = -r_{ji}$ we get the nullity for $u_{ks} \wedge r_{ij}$.
\end{enumerate}

\paragraph{Image of $M \wedge M$.}
First we will do a small detour through  bilinear forms on matrices. 
Until the very end we work on an arbitrary field $\mathbb{K}$.  
Recall that $\mathfrak M_n(\mathbb K)$ denotes the $\mathbb K$-vector space of $n$-by-$n$ matrices with entries in 
$\mathbb K$ and $\mathbf 1_n\in \mathfrak M_n(\mathbb K)$ the identity matrix. 
Note that, if  $i \neq j$, 
the inverse of $\mathbf 1_n + e_{ij}$ is $\mathbf 1_n - e_{ij}$ and 
that elementary matrices multiply according to the rule  
$e_{ij}e_{st} = \delta_{js}e_{it}$. We start by a very  classical result:

\begin{lemma}\label{lem duality}
Let ${\rm tr}$ denote the trace map. 
Then for any field $\mathbb K$ and any integer $n$ the homomorphism:
\[
\begin{array}{rcl}
 \mathfrak M_n(\mathbb K) & \rightarrow & {\rm Hom}(\mathfrak M_n(\mathbb K), \mathbb K) \\
 A & \mapsto & B \leadsto {\rm tr}(AB)
\end{array}
\]
is an isomorphism.
\end{lemma}

\vspace{0.2cm}\noindent
A little less classical is:
\begin{theorem} Let $n \geq 2$.
The $\mathbb{K}$-vector space    
$Hom_{GL(n, \mathbb{K})}(\mathfrak M_n(\mathbb{K}),\mathfrak M_n(\mathbb{K}))$  
has dimension $2$. It is generated by  the identity map $Id_{\mathfrak M_n(\mathbb{K})}$ 
and by the map $\Psi(M) = tr(M)\mathbf 1_n$. 
\end{theorem}
\begin{proof}
It is easy to check that the two equivariant maps    
$Id_{\mathfrak M_n(\mathbb{K})}$ and $\Psi$ are linearly independent. 
Indeed, evaluating a linear dependence relation   
$\alpha Id_{\mathfrak M_n(\mathbb{K})} + \beta \Psi =0$ on 
$e_{12}$ one gets $\alpha = 0 = \beta$.

\vspace{0.2cm}\noindent 
Fix an arbitrary $\phi \in {\rm Hom}_{GL(n, \mathbb{K})}(\mathfrak M_n(\mathbb{K}),\mathfrak M_n(\mathbb{K}))$. Denote by $A = (a_{ij})$ the matrix $\phi(e_{11})$.
Consider two integers $ 1 < s\neq t \leq n$. 
From the equality $(\mathbf 1_n + e_{st})e_{11}(\mathbf 1_n - e_{st}) = e_{11}$ we deduce that:
\begin{eqnarray}
 \phi(e_{11}) & = & (\mathbf 1_n + e_{st})\phi(e_{11})(\mathbf 1_n - e_{st}) \\
                       & = & \phi(e_{11}) + e_{st}\phi(e_{11}) - \phi(e_{11})e_{st} - e_{st}\phi(e_{11})e_{st} \\
                       & = & \phi(e_{11}) + \sum_{1 \leq j \leq n} a_{tj}e_{sj} - \sum_{1 \leq i \leq n} a_{is}e_{it} - a_{ts}e_{st}
\end{eqnarray}
Therefore, for $ 1 < s\neq t \leq n$ we have:
\begin{eqnarray}
 \sum_{1 \leq j \leq n} a_{tj}e_{sj} - \sum_{1 \leq i \leq n} a_{is}e_{it} - a_{ts}e_{st} & = & 0
\end{eqnarray}
The first term in this sum is a matrix with only one non-zero row, 
the second a matrix with only one non-zero column and the third  a matrix 
with a single (possibly) non-zero entry. The only common entry 
for this three matrices appears for  $j = t$ and $i =s$, where 
we get the equation  $a_{tt} - a_{ss} - a_{ts} = 0$. Otherwise, 
$a_{tj} = 0$, for all $j\neq t$, and $a_{is} = 0$, for all $i\neq s$. 
Observe that, in particular,  $a_{ts} = 0$. Summing up, 
either  in the column $s$ or in the row 
$t$ of the matrix $A$, the only possible non-zero 
elements are those that appear in the equation $a_{tt} - a_{ss} = 0$.

\vspace{0.2cm}\noindent 
Letting  $s$ and $t$ vary, one deduces  that  
$a_{ts}=0$, and 
$A= \phi(e_{11})$ is of the form  
$\lambda e_{11} + \mu \sum_{i= 2}^n e_{ii}$ 
for two scalars $\lambda, \mu \in \mathbb{K}$.

\vspace{0.2cm}\noindent 
Let $T_{ij}$ be the invertible matrix that interchanges the basis vectors $i$ and $j$. Then $T_{ij}e_{ii}T_{ij} = e_{jj}$, $T_{ij}e_{jj}T_{ij}=e_{ii}$ and $T_{ij}e_{kk}T_{ij} = e_{kk}$ for $k  \neq i$ and $k \neq j$. Therefore, 
$\phi(e_{jj}) = \phi(T_{1j}e_{11}T_{1j}) =  T_{1j}\phi(e_{11})T_{1j}.$ 
And from the description of  $\phi(e_{11})$ one gets: 
\[
\phi(e_{jj}) = \lambda e_{jj} + \sum_{i \neq j} \mu e_{ii}, \; {\rm for \; all} \; 1 \leq j \leq n.
\]

\vspace{0.2cm}\noindent 
From the relation $(\mathbf 1_n + e_{ij})e_{ii}(\mathbf 1_n - e_{ij}) = e_{ii} - e_{ij}$ we get:
\begin{eqnarray}
 \phi((\mathbf 1_n + e_{ij})e_{ii}(\mathbf 1_n - e_{ij})) & = & \phi(e_{ii}) - \phi(e_{ij}) \\
 & =& (\mathbf 1_n + e_{ij})\phi(e_{ii})(\mathbf 1_n - e_{ij}) \\
 & = & \phi(e_{ii}) + e_{ij}\phi(e_{ii})   - \phi(e_{ii})e_{ij} -e_{ij}\phi(e_{ii})e_{ij},
\end{eqnarray}
and in particular:
\begin{eqnarray}
 \phi(e_{ij}) & = & -  e_{ij}\phi(e_{ii})   +  \phi(e_{ii})e_{ij}  + e_{ij}\phi(e_{ii})e_{ij} \\
 & = & - \mu e_{ij} + \lambda e_{ij} \\
 & = & (\lambda - \mu) e_{ij}
\end{eqnarray}
This shows that   $\phi$ is completely determined by  $\phi(e_{ii})$ and that    ${\rm Hom}_{GL_n(\mathbb{K})}(\mathfrak M_n(\mathbb{K}),\mathfrak M_n(\mathbb{K}))$ 
has dimension at most  $2$.  Observe that the identity map  
$Id_{\mathfrak M_n(\mathbb{K})}$ corresponds to $\lambda  = 1, \mu=0$ 
and that  $\Psi$ corresponds to $\lambda = \mu = 1$.
\end{proof}

\vspace{0.2cm}\noindent 
From these two results we deduce the result that will save us from lengthy computations:
\begin{proposition}\label{prop equivbilform}
Let $n \geq 2$.
For any field $\mathbb{K}$ the vector space of bilinear forms on 
$\mathfrak M_n(\mathbb{K})$ invariant under conjugation by 
$GL(n, \mathbb{K})$ has as basis the bilinear maps 
$(A,B) \mapsto tr(A) tr(B)$ and $(A,B) \mapsto tr(AB)$. 
If $char (\mathbb{K}) \neq 2$, the subspace of alternating 
bilinear forms is trivial. If  $char (\mathbb{K}) = 2$ the 
space of bilinear alternating forms is generated by the 
form $(A,B) \mapsto tr(A)tr(B) + tr(AB)$.
\end{proposition}

\vspace{0.2cm}\noindent 
Consider now the  canonical map
\begin{equation}
 \wedge^2 (\mathfrak M_{g}(\Z/p\Z)) \rightarrow \wedge^2 (\mathfrak M_{g}(\Z/p\Z))_{GL_g(\Z/p\Z)} \rightarrow \wedge^2(\mathfrak{sp}_{2g}(p))_{GL(g, \Z/p\Z)} \rightarrow \wedge^2(\mathfrak{sp}_{2g}(p))_{Sp(2g,\Z/p\Z)}.
\end{equation}
By construction its image is the span of the image of $M \wedge M$ 
in $\wedge^2(\mathfrak{sp}_{2g}(p))_{Sp(2g,\Z/p\Z)}$. By 
Proposition \ref{prop equivbilform}, it is $0$ if $p$ is odd and it is at 
most $\Z/2\Z$ if $p=2$. 

\vspace{0.2cm}\noindent 
Further, the unique $GL(g, \Z/2\Z)$-invariant 
alternating form on $\mathfrak M_g(\Z/2\Z)$ does not vanish on the element 
$e_{11}\wedge e_{22}$ and hence we have a nontrivial element of 
$\wedge^2 (\mathfrak M_{g}(\Z/2\Z))_{GL(g, \Z/2\Z)}$, which is a 
submodule of $\wedge^2(\mathfrak{sp}_{2g}(2))_{GL(g, \Z/2\Z)}$.  
Note that the image of $e_{11}\wedge e_{22}$ in $\wedge^2(\mathfrak{sp}_{2g}(2))_{GL(g, \Z/2\Z)}$
is the class of the element $n_{11} \wedge n_{22}$.

\vspace{0.2cm}\noindent 
Furthermore, fix a symplectic basis $\{a_i,b_i\}_{1 \leq i \leq g}$ of $\Z^{2g}$. 
Then, the $2g+1$ transvections along the elements $a_{i},b_{j}-b_{j+1}$ 
for $1 \leq i \leq g$ and $1 \leq j \leq g-1$, $b_{g-1}$ and $b_g$ 
generate $Sp(2g,\Z/2\Z)$, for instance because they are the canonical 
images of the set of Dehn twists 
generators of the mapping class group considered by Humphries in \cite{Hump}. One checks 
directly, using all elements we know they vanish in 
$\wedge^2(\mathfrak{sp}_{2g}(2))_{Sp(2g,\Z/2\Z)}$, that the action 
of these generators on $n_{11} \wedge n_{22}$ is trivial and hence it defines an 
element of $\wedge^2(\mathfrak{sp}_{2g}(2))_{Sp(2g,\Z/2\Z)}$. 
This finishes our proof.

\begin{remark} It is clear from the proof that the copy $\Z/2\Z$ we have 
detected is stable, in the sense that the homomorphism 
$ \wedge^2(\mathfrak{sp}_{2g}(2))_{Sp(2g,\Z/2\Z)} \rightarrow \wedge^2(\mathfrak{sp}_{2g+2}(2))_{Sp(2g+2,\Z/2\Z)}$ is an isomorphism, for all $g\geq 3$, since 
both are detected by the obvious stable element $n_{11} \wedge n_{22}.$
\end{remark}


\appendix
\section{Appendix: Weil representations using theta functions}
\subsection{Weil representations at level $k$, for even $k$ following \cite{F1,F2,Go}} 
Let ${\cal S}_{g}$ be the Siegel space of $g\times g$ symmetric matrices $\Omega$
of complex entries having the imaginary part ${\rm Im}\, \Omega$ positive defined.
We represent any element $\gamma \in Sp(2g,{\Z})$ as
\begin{math}
\left ( \begin{array}{cc}
   A & B\\
   C & D
\end{array}  \right)
\end{math}
where  $A, B, C, D$ are $g\times g$ matrices.
There is a natural $Sp(2g,{\Z})$ action on ${\C}^{g}\times {\cal S}_{g}$ given by
\begin{equation}
\gamma \cdot (z,\Omega)=((((C\Omega+D)^{\top})^{-1})z,(A\Omega+B)(C\Omega+D)^{-1}).
\end{equation} 
The dependence of the classical theta function
$\theta(z,\Omega)$ on $\Omega$ is expressed by a functional equation
which describes its behavior under the
action of $Sp(2g,{\Z})$.
Let $\Gamma(1,2)$
be the so-called theta group consisting of elements
$\gamma\in Sp(2g,{\Z})$
which preserve the quadratic form
\[ Q(n_1,n_2,...,n_{2g})= \sum_{i=1}^gn_in_{i+g} \in {\Z}/2{\Z},\]
which means that $Q(\gamma(x))=Q(x)({\rm mod} ~ 2)$.
Then $\Gamma(1,2)$  may be alternatively described as the set
of those elements $\gamma$ having the  property that the diagonals of $A^{\top}C$ and
$B^{\top}D$  are even.
Let  $\langle , \rangle $ denote the standard hermitian product on  ${\C}^{2g}$. The functional
equation, as stated in \cite{Mum3} is:
\[ \theta({(C\Omega+D)^{\top}}^{-1}z, (A\Omega+B)(C\Omega+D)^{-1})=
        \zeta_{\gamma}  {\det} (C\Omega+D)^{1/2}\exp(\pi\sqrt{-1} \langle z,(C\Omega+D)^{-1}Cz \rangle )\theta(z,\Omega), \] 
for $\gamma \in \Gamma(1,2)$,   where $\zeta_{\gamma}$ is a certain 
 $8^{th}$ root of unity.

\vspace{0.2cm}\noindent
If $g=1$ we may suppose that $C>0$ or $C=0$ and $D>0$ so we have 
${\rm Im}(C\Omega +D)\geq 0$ for $\Omega$ in the upper half plane. Then we will 
choose
the square root $(C\Omega +D)^{1/2}$ in the first quadrant. Now we can express
the dependence of $\zeta_{\gamma}$ on $\gamma$ as follows:
\begin{enumerate}
  \item   for even $C$  and  odd $D$,  
$\zeta_{\gamma}=\sqrt{-1}^{(D-1)/2}(\frac{C}{\mid D\mid})$, 
  \item   for odd $C$ and  even $D$,  $\zeta_{\gamma}=\exp(-\pi \sqrt{-1}C/4)(\frac{D}{C})$,

\end{enumerate}
where $(\frac{x}{y})$ is the usual Jacobi symbol, see \cite{Har-Wr}.

\vspace{0.2cm}\noindent
For $g>1$ it is less obvious to describe this dependence.
 We fix first the choice of the square root
of $ {\det} (C\Omega+D)$ in the following manner: let $ {\det} ^{\frac{1}{2}}\left(\frac{Z}{\sqrt{-1}}\right)$ be the
unique holomorphic function on ${\cal S}_g$ satisfying
\[ \left( {\det} ^{\frac{1}{2}}\left(\frac{Z}{\sqrt{-1}}\right)\right)^2= {\det} \left(\frac{Z}{\sqrt{-1}}\right), \]
and taking in $\sqrt{-1}{\bf 1}_g$ the value 1. Next define
\[  {\det} ^{\frac{1}{2}}\left(C\Omega+D\right)= {\det} ^{\frac{1}{2}}(D) {\det} ^{\frac{1}{2}}\left(\frac{\Omega}{\sqrt{-1}}\right)  {\det} ^{\frac{1}{2}}\left(\frac{-\Omega^{-1}-D^{-1}C}{\sqrt{-1}}\right), \]
where the square root of $ {\det} (D)$ is taken to lie in the first
quadrant. Using this convention we may express $\zeta_{\gamma}$ as a
Gauss sum for invertible $D$, see \cite[pp. 26-27]{Frei}: 
\begin{equation}
\zeta_{\gamma}= {\det} ^{-\frac{1}{2}}(D)\sum_{\ell\in {\Z}^g/\!D{\Z}^g}\exp(\pi \sqrt{-1} \langle \ell,BD^{-1}\ell \rangle), \\
\end{equation}
and in particular we recover the formula from above for $g=1$.
On the other hand  for
$
\gamma=
\left ( \begin{array}{cc}
    A & 0 \\
    0 & (A^{\top})^{-1}
\end{array}  \right)
$
we have  $\zeta_{\gamma}=( {\det}  A)^{-1/2}$.
We recall that a multiplier system (\cite{Frei}) for a subgroup $\Gamma\subset
\mbox{\rm Sp}(2g,{\R})$ is a map $m:\Gamma\longrightarrow {\C}^*$ such that
\[ m(\gamma_1\gamma_2)=s(\gamma_1,\gamma_2)m(\gamma_1)m(\gamma_2).\]
An easy remark is that, once a multiplier system $m$ is chosen,  the
product $A(\gamma,\Omega)=m(\gamma)j(\gamma,\Omega)$ verifies the
cocycle condition
\[
A(\gamma_1\gamma_2,\Omega)=A(\gamma_1,\gamma_2\Omega)A(\gamma_2,\Omega),
\]
for $\gamma_i\in \Gamma$. Then another formulation of the 
dependence of $\zeta_{\gamma}$ on $\gamma$ is to say that it is the 
multiplier system defined on $\Gamma(1,2)$. Remark that using the
congruence subgroup property due to Mennicke  (\cite{Me1,Me2}) and 
Bass, Milnor and Serre (\cite{BMS}) any two multiplier systems defined on a subgroup
of the theta group are identical on some congruence subgroup. 

\vspace{0.2cm}\noindent
 When 
$
\gamma=
\left ( \begin{array}{cc}
    1_g & B \\
    0 & 1_g
\end{array}  \right)
$
then the multiplier system is trivial, $\zeta_{\gamma}=1$, and eventually for
$
\gamma=
\left ( \begin{array}{cc}
    0 & -1_g \\
    1_g &  0
\end{array}  \right)
$
we have $\zeta_{\gamma}=\exp(\pi \sqrt{-1}g/4)$. 
Actually  this data determines completely $\zeta_{\gamma}$.

\vspace{0.2cm}\noindent
Denote  
$ {\det} ^{\frac{1}{2}}(C\Omega+D)=j(\gamma,\Omega)$. Then there exists  a map 

\[ s:Sp(2g,{\R})\times Sp(2g,{\R})\longrightarrow \{-1,1\}\]
satisfying 
\[ j(\gamma_1\gamma_2, \Omega)=s(\gamma_1,\gamma_2)j(\gamma_1,\gamma_2
\Omega) j(\gamma_2,\Omega). \]

\vspace{0.2cm}\noindent
Consider now the level $k$ theta functions. For $m \in ({\Z}/k{\Z})^g$   these are defined by
\begin{equation}
  \theta_{m}(z,\Omega)=\sum_{\ell\in m+k{\Z}^g} \exp\left(\frac{\pi
      \sqrt{-1}}{k}
\left( \langle \ell,\Omega \ell \rangle +2 \langle \ell,z \rangle \right)\right)
\end{equation}
or, equivalently, by
\[ \theta_{m}(z,\Omega)= \theta (m/k,0)(kz,k\Omega). \]
where $\theta(*,*)$ are the theta functions with rational characteristics
(\cite{Mum3}) given by
\begin{equation}
\theta(a,b)(z,\Omega) = \sum_{\ell\in {\Z}^g}\exp\left(\frac{\pi
    \sqrt{-1}}{k}
\left( \langle \ell+a,\Omega(\ell+a) \rangle +2 \langle \ell+a,z+b \rangle \right)\right)
\end{equation}
for $a,b\in {\Q}^g$.
Obviously $\theta(0,0)$  is the usual theta function.

\vspace{0.2cm}\noindent
Let us  denote by $R_{8}\subset {\C}^\ast$ the group of $8^{th}$ roots of unity. Then
$R_{8}$ becomes also a subgroup of the unitary group $U(n)$ acting by scalar
multiplication.
Consider also the  theta vector of level $k$:
\[ \Theta_{k}(z,\Omega)=(\theta_{m}(z,\Omega))_{m\in({\Z}/k{\Z})^{g}}. \]

\begin{proposition}[\cite{F1,F2,Go}]
The theta vector satisfies the following functional equation:
\begin{equation}
\Theta_k (\gamma\cdot(z,\Omega))  = \zeta_{\gamma}  {\det} (C\Omega+D)^{1/2}
\exp(k\pi \sqrt{-1}\langle z,(C\Omega+D)^{-1}Cz \rangle) \rho_{g}(\gamma)(\Theta_{k}(z,\Omega))
\end{equation}
where
\begin{enumerate}      
 \item  $\gamma$ belongs to the theta group $\Gamma(1,2)$ if $k$ is odd
and to $\mbox{\rm Sp}(2g,{\Z})$ elsewhere.
 \item   $\zeta_{\gamma} \in R_{8}$ is the (fixed) multiplier system 
described above.
 \item   $\rho_{g}: \Gamma(1,2)\longrightarrow U(\C^{(\Z/k\Z)^g})$ is a group 
homomorphism. For even $k$ the corresponding
map $\rho_{g}:\mbox{\rm Sp}(2g,{\Z}) \longrightarrow U(\C^{(\Z/k\Z)^g}))$ 
becomes a group homomorphism (denoted also by $\rho_g$ when no confusion 
arises) when passing to the quotient $U(\C^{(\Z/k\Z)^g})/R_8$. 
 \item  $\rho_g$ is determined by the points {\rm (1-3)} above. 
\end{enumerate}
\end{proposition}

\begin{remark}
This result is stated also in \cite{Igusa} for some modified theta functions
but  in  less  explicit  form. 
\end{remark}

\subsection{Linearizability of Weil representations for odd level $k$}
The proof for the linearizability of  the 
Weil representation associated to $\Z/k\Z$ 
for odd $k$ was first given by A. Andler 
(see \cite{AR}, Appendix A III) and then extended 
to other local rings in \cite{CMS}. 
Let $\eta: Sp(2g, \Z)\times Sp(2g,\Z)\to R_8\subset U(1)$
be the cocycle determined by the Weil representation associated to 
$\Z/k\Z$. The image $S^2$  of  $\left ( \begin{array}{cc}
    -1_g & 0  \\
    0  &  -1_g 
\end{array} \right )$ is the involution   
\[S^2 \theta_m = \theta_{-m},\; m\in (\Z/k\Z)^g.\]
Thus $S^4=1$ and $S^2$ has eigenvalues $+1$ and $-1$. 
Moreover $S^2$ is central and hence the Weil representation 
splits according to the eigenspaces decomposition.  
Further, the determinant of each factor representation 
is a homogeneous function  whose degree is the respective 
dimension of the factor. Therefore we could express,  
for each one of the two factors,  
$\eta$ to the power the dimension of the respective factor 
as a determinant cocycle.  
The difference between the two factors' 
dimensions is the trace of $S^2$, namely $1$ 
for odd $k$ and $2^g$ for even $k$. This implies that 
$\eta$, for odd $k$, and $\eta^{2^g}$, for even $k$ is 
a boundary cocycle. However $\eta^8=1$ and hence 
for even $k$ and $g\geq 3$ this method could not give any 
non-trivial information about $\eta$.

{
\small

\bibliographystyle{plain}

}

\end{document}